\documentclass[11pt]{amsart}
\usepackage{a4wide,enumerate}
\allowdisplaybreaks

\let\pa\partial  
\let\na\nabla  
\let\eps\varepsilon  
\newcommand{\N}{{\mathbb N}}  
\newcommand{\R}{{\mathbb R}} 
\newcommand{\diver}{\operatorname{div}}  

\newcommand{\HH}{{\mathcal H}}

\newtheorem{theorem}{Theorem}   
\newtheorem{lemma}[theorem]{Lemma}   
   
\newtheorem{remark}[theorem]{Remark}

\begin{document}  

\title[The boundedness-by-entropy principle]{The boundedness-by-entropy principle \\
for cross-diffusion systems}

\author{Ansgar J\"ungel}
\address{Institute for Analysis and Scientific Computing, Vienna University of  
	Technology, Wiedner Hauptstra\ss e 8--10, 1040 Wien, Austria}
\email{juengel@tuwien.ac.at} 

\date{\today}

\thanks{The author acknowledges partial support from   
the Austrian Science Fund (FWF), grants P22108, P24304, and W1245, and    
the Austrian-French Program of the Austrian Exchange Service (\"OAD).
Part of this manuscript was written during the stay of the author at the
King Abdullah University of Science and Technology (KAUST) in Thuwal,
Saudi-Arabia. The author thanks Peter Markowich for his kind invitation and support} 

\begin{abstract}
A novel principle is presented which allows for the proof of bounded weak
solutions to a class of 
physically relevant, strongly coupled parabolic systems exhibiting
a formal gradient-flow structure. The main feature of these systems is that the
diffusion matrix may be generally neither symmetric nor positive semi-definite.
The key idea of the principle is to employ a transformation of variables, determined
by the entropy density, which is defined by the gradient-flow formulation.
The transformation yields at the same time a positive semi-definite diffusion matrix,
suitable gradient estimates as well as lower and/or
upper bounds of the solutions. These bounds are a consequence of the
transformation of variables and are obtained without the use of a maximum
principle. Several classes
of cross-diffusion systems are identified which can be solved by this technique. 
The systems are formally derived from continuous-time random walks on a lattice
modeling, for instance, the motion of ions, cells, or fluid particles.
\end{abstract}

 

\maketitle
\tableofcontents


\section{Introduction}\label{sec.intro}

Many applications in physics, chemistry, and biology can be modeled by
reaction-diffusion systems with cross diffusion,
which describe the temporal evolution of the densities or mass fractions
of a multicomponent system. Physically, we expect that the concentrations
are nonnegative or even bounded (examples are given in Section \ref{sec.ex}). 
Since generally no maximum principle
holds for parabolic systems, the proof of these bounds
is a challenging problem. A second difficulty arises from the fact that
in many applications, the diffusion matrix is neither symmetric nor
positive semi-definite. 

In this paper, we introduce a general technique which
allows us, under certain assumptions, to prove simultaneously the global existence
of a weak solution as well as its boundedness from below and/or above.
The key idea is to exploit the so-called entropy structure of the parabolic system,
which is assumed to exist,
leading at the same time to gradient estimates and lower/upper bounds.
More specifically, we consider reaction-diffusion systems of the form  
\begin{equation}\label{1.eq}
  \pa_t u - \diver(A(u)\na u) = f(u) \quad\mbox{in }\Omega,\ t>0,
\end{equation}
subject to the boundary and initial conditions
\begin{equation}\label{1.bic}
  (A(u)\na u)\cdot\nu=0\quad\mbox{on }\pa\Omega,\ t>0, \quad u(0)=u^0
	\quad\mbox{in }\Omega.
\end{equation}
Here, $u(t)=(u_1,\ldots,u_n)(\cdot,t)$ 
is a vector-valued function ($n\ge 1$), 
representing the densities or mass fractions $u_i$ of the components of the system, 
$A(u)=(a_{ij}(u))\in\R^{n\times n}$ is the diffusion matrix, and the reactions
are modeled by the components of the function $f:\R^n\to\R^n$.
Furthermore, $\Omega\subset\R^d$ ($d\ge 1$) is a bounded domain
with Lipschitz boundary and $\nu$ is the exterior unit normal vector to
$\pa\Omega$. The divergence $\diver(A(u)\na u)$ and the expression
$(A(u)\na u)\cdot\nu$ consist of the components 
$$
  \sum_{j=1}^d\sum_{k=1}^n\frac{\pa}{\pa x_j}\left(a_{ik}(u)\frac{\pa u_k}{\pa x_j}
	\right), \quad 
	\sum_{j=1}^d\sum_{k=1}^n a_{ik}(u)\frac{\pa u_k}{\pa x_j}\nu_j, \quad
	i=1,\ldots,n,
$$
respectively. From the applications, we expect that the solution $u$
has values in an open set $D\subset\R^n$. When $u$ models concentrations,
we expect that $D\subset(0,\infty)^n$ (positivity); when $u$ models mass
fractions, we expect that the values of each component $u_i$ are between
zero and one, i.e.\ $D\subset(0,1)^n$ (boundedness and positivity).

Before summarizing the state of the art of cross-diffusion systems and
detailing our main results, we explain the key idea of our principle
and provide an illustrating example.

\subsection{Idea of the principle}\label{sec.idea}

The main assumption in this paper is that system \eqref{1.eq}
possesses a formal gradient-flow structure, i.e., \eqref{1.eq} can be formulated as
$$
  \pa_t u - \diver\left(B\na \frac{\delta\HH}{\delta u}\right) = f(u),
$$
where $B$ is a positive semi-definite matrix and
$\delta\HH/\delta u$ is the variational derivative of the entropy
(or free energy) functional $\HH[u]=\int_\Omega h(u)dx$. 
The function $h:D\to[0,\infty)$ is called the
entropy density. Identifying $\delta\HH/\delta u$ with its Riesz
representative $Dh(u)$ (the derivative of $h$) and introducing the 
entropy variable $w=Dh(u)$, the above formulation can be understood as
\begin{equation}\label{1.eqw}
  \pa_t u - \diver(B(w)\na w) = f(u),
\end{equation}
where $B=B(w)=A(u)(D^2 h(u))^{-1}$ and $D^2h$ is the Hessian of $h$.
For transforming back from the $w$- to the $u$-variable, we need to assume that
$Dh:D\to\R^n$ is invertible such that $u=(Dh)^{-1}(w)$.

The gradient-flow formulation has two important consequences.
First, calculating the formal time derivative of $\HH$, \eqref{1.eqw} and
integrating by parts yields
\begin{equation}\label{1.dHdt}
  \frac{d\HH}{dt} = \int_\Omega \pa_t u\cdot Dh(u)dx
	= \int_\Omega \pa_t u\cdot wdx = -\int_\Omega \na w:B(w)\na w dx 
	+ \int_\Omega f(u)\cdot wdx,
\end{equation}
where $A:B=\sum_{i,j}a_{ij}b_{ij}$ for matrices $A=(a_{ij})$ and $B=(b_{ij})$.
Thus, if $f(u)\cdot w\le 0$ and since $B(w)$ is assumed to be positive semi-definite, 
$\HH$ is a Lyapunov functional. We refer to $\HH$ as an entropy and to the
integral of $\na w:B\na w$ as the corresponding entropy dissipation. 
Under certain conditions, it gives gradient 
estimates for $u$ needed to prove the global-in-time existence
of solutions to \eqref{1.bic} and \eqref{1.eqw}.

We remark that the positive semi-definiteness of $B(w)$ is in fact a consequence
of the existence of an entropy. It was shown in \cite{DGJ97,KaSh88} 
that both properties are equivalent and moreover, $B(w)$ may be even symmetric.

Second, supposing that there exists a weak solution $w$ to \eqref{1.eqw},
the invertibility of $Dh:D\to\R^n$ shows that the original 
variable $u=(Dh)^{-1}(w)$ satisfies $u(\cdot,t)\in D$ for $t>0$.
Thus, if $D$ is bounded, we automatically obtain $L^\infty$ bounds for $u$,
without using a maximum principle. If $D$ is only a cone, for instance
$D=(0,\infty)^n$, we conclude the positivity of $u(t)$.

We call the above technique the {\em boundedness-by-entropy principle} 
since it provides lower and/or upper bounds for the solutions to 
\eqref{1.eq}-\eqref{1.bic} by the use of the entropy density.
Summarizing, the principle is based on two main hypotheses:

\begin{description}
\item[\rm H1] There exists a function $h\in C^2(D;[0,\infty))$ whose derivative
is invertible on $\R^n$. This yields the bound $u(\cdot,t)\in D$.
\item[\rm H2] The matrix $D^2h(u)A(u)$ is positive semi-definite for all $u\in D$.
This condition is necessary to derive a priori estimates for $u$.
\end{description}

Note that the positive semi-definiteness of $D^2h(u)A(u)$ is equivalent
to that of $B(w)=A(u)(D^2 h(u))^{-1}$ since for all $z\in\R^n$,
$z^\top D^2h(u)A(u)z=(D^2h(u)z)^\top B(w)(D^2h(u)z)$. Hypothesis H2 avoids
the inversion of $D^2h(u)$.

In fact, we need a stronger hypothesis than H2 since it does not allow us
to infer gradient estimates. We need to suppose that $D^2h(u)A(u)$ is 
positive definite in such a way that we obtain
$L^2$ gradient estimates for $u_i^{m}$, where $m>0$ is some number.
Moreover, we need an estimate for the time derivative of $u_i$
which makes it necessary to impose some growth conditions on the coefficients
of $A(u)$. We explain these requirements with the help of the following example.

\subsection{An illustrative example}\label{sec.illust}

We consider a multicomponent fluid
consisting of three components with mass fractions $u_1$, $u_2$, and 
$1-u_1-u_2$ and equal molar masses under isobaric, isothermal conditions. 
The model equals \eqref{1.eq} with the diffusion matrix
\begin{equation}\label{1.ms}
  A(u) = \frac{1}{2+4u_1+u_2}\begin{pmatrix}
	1+2u_1 & u_1 \\ 2u_2 & 2+u_2
	\end{pmatrix}
\end{equation}
where we have chosen particular diffusivities to simplify the presentation
(see Section \ref{sec.ms}). 
Notice that the nonnegativity
of $u_1$ and $u_2$ can be proved easily by a maximum principle argument
but the proof of upper bounds is less clear.
The logarithmic entropy density 
\begin{equation}\label{1.log}
  h(u) = u_1(\log u_1-1)+u_2(\log u_2-1) + (1-u_1-u_2)(\log(1-u_1-u_2)-1)
\end{equation}
for $u=(u_1,u_2)\in D$, where
\begin{equation}\label{1.D}
  D=\{(u_1,u_2)\in(0,1)^2:u_1+u_2<1\}
\end{equation}
satisfies Hypothesis H1, since the
inverse transformation of variables gives
\begin{equation}\label{1.u}
  u = (Dh)^{-1}(w) = \left(\frac{e^{w_1}}{1+e^{w_1}+e^{w_2}},
	\frac{e^{w_2}}{1+e^{w_1}+e^{w_2}}\right)\in D,
\end{equation}
where $w=(w_1,w_2)\in\R^2$. Thus, once the existence of weak solutions $w$
to the transformed system \eqref{1.eqw} is shown, we conclude the bounds
$0<u_1,u_2<1$ {\em automatically} by transforming back to the original variable.
In particular, no maximum principle is used. In fact,
the inverse transformation even shows that $1-u_1-u_2>0$,
as required from the application. Furthermore, the matrix
$$
  D^2h(u)A(u) =\frac{1}{\delta(u)}	\begin{pmatrix}
	u_2(1-u_1-u_2)+3u_1u_2 & 3u_1u_2 \\ 3u_1u_2 & 2u_1(1-u_1-u_2)+3u_1u_2
	\end{pmatrix}
$$
with $\delta(u)=u_1u_2(1-u_1-u_2)(2+4u_1+u_2)$
is symmetric and positive definite, thus fulfilling Hypothesis H2.

However, the matrix $B=A(u)(D^2h(u))^{-1}$ degenerates at $u_1=0$ or $u_2=0$,
and we cannot expect to conclude gradient estimates for $w$ from \eqref{1.dHdt}.
Instead, it is more appropriate to derive these estimates for the original 
variable $u$. Indeed, calculating the time derivative of the entropy according to
\eqref{1.dHdt} and using $\na w=D^2h(u)\na u$, we find that
$$
  \frac{d\HH}{dt} + \int_\Omega\na u:D^2h(u)A(u)\na u dx = \int_\Omega f(u)\cdot wdx,
$$
and the entropy dissipation can be estimated according to
\begin{align*}
  \int_\Omega \na u:D^2h(u)A(u)\na u dx
	&= \int_\Omega\frac{1}{2+4u_1+u_2}\left(
	\frac{|\na u_1|^2}{u_1}+\frac{2|\na u_2|^2}{u_2}
	+ \frac{3|\na(u_1+u_2)|^2}{1-u_1-u_2}\right)dx \\
	&\ge \int_\Omega\big(2|\na\sqrt{u_1}|^2 + 4|\na\sqrt{u_2}|^2\big)dx.
\end{align*}
Thus, assuming that the integral involving the
reaction terms can be bounded uniformly in $u$,
we obtain $H^1(\Omega)$ estimates for $u_1^{1/2}$ and $u_2^{1/2}$.
Using the boundedness of $u_i$, a priori estimates for $\pa_t u_1$ and
$\pa_t u_2$ can be proven, taking into account the particular structure of $A(u)$.

This example shows that we need additional assumptions on the nonlinearities
of system \eqref{1.eq} in order to derive suitable a priori estimates,
detailed in Section \ref{sec.main}.

\begin{remark}[Notion of entropy]\label{rem.ent}\rm
There exists an intimate relation between the boundedness-by-entropy principle 
and non-equilibrium thermodynamics. In particular, the entropy variable $w=Dh(u)$ 
is related to the chemical potentials of a mixture of gases 
and the special transformation \eqref{1.u} is connected with a special choice of 
thermodynamic activities; see Appendix \ref{sec.thermo} for details.
The entropy density defined above equals the negative thermodynamic entropy.
Since the physical entropy is increasing and we wish to investigate
nonincreasing functionals, we have reversed the sign as usual in entropy methods. 
Moreover, the logarithmic entropy \eqref{1.log} is motivated by Boltzmann's 
entropy for kinetic equations. For these reasons, we refer to the
functional $\HH[u]=\int_\Omega h(u)dx$ as a (mathematical) entropy. 
We note that in some applications, free energy may be a more appropriate notion. 
\qed
\end{remark}

\subsection{State of the art}\label{sec.state}

We have already mentioned that the analysis of cross-diffusion
systems is delicate since standard tools like maximum principles and 
regularity results generally do not apply. For instance, there exist H\"older
continuous solutions to certain cross-diffusion systems which are not bounded,
and there exist bounded weak solutions which develop singularities in finite time
\cite{StJo95}. In view of these counterexamples, it is not surprising that 
additional conditions are required to prove that (local in time)
weak solutions are bounded and that they can be continued globally in time.

Ladyzenskaya et al.\ \cite[Chap.~VII]{LSU88} reduce the problem of finding
a priori estimates of local-in-time solutions $u$ to quasilinear parabolic systems
to the problem of deriving $L^\infty$ bounds for $u$ and
$\na u$. Under some growth conditions on the nonlinearities,
the global-in-time existence of classical solutions was shown.
A fundamental theory of strongly coupled systems was developed by Amann \cite{Ama89}.
He formulated the concept of $W^{1,p}$ weak solutions and their local existence 
and proved that the solutions exist globally if their $L^\infty$ and H\"older norms 
can be controlled. The above mentioned counterexamples show that the control on
both norms is necessary. Le and Nguyen \cite{LeNg06} proved that bounded weak 
solutions are H\"older continuous if certain structural assumptions on the
diffusion matrix are imposed.
The regularity of the solutions to systems with diagonal or full diffusion matrix
was investigated in, for instance, \cite{Ama93,DMS11,Lun95,Pru03}.

The boundedness of weak solutions to strongly coupled systems has been proved  
using various methods. Invariance principles were employed by K\"ufner \cite{Kue96}
and Redlinger \cite{Red89}, requiring severe restrictions on the initial data.
Truncated test functions, which are nonlinear in the solutions to $2\times 2$
systems, were suggested by Le \cite{Le06} who proved
the boundedness under rather restrictive structural assumptions.
In the work of Lepoutre et al.\ \cite{LPR12}, 
the existence of bounded solutions to strongly coupled systems with spatially
regularized arguments in the nonlinearities was shown using H\"older theory
for nondivergence parabolic operators.
Other methods are based on the derivation of $L^p$ bounds uniform in $p$ and
the passage to the limit $p\to\infty$ (Moser-type or Alikakos-type
iterations \cite{Ali79}). 

The idea of proving the boundedness of weak solutions using the entropy density
\eqref{1.log} was, to our best knowledge, first employed by Burger et al.\ 
\cite{BDPS10} in a size-exclusion model for two species
(see Section \ref{sec.burger}).
It was applied to a tumor-growth model in \cite{JuSt12} and extended to
Maxwell-Stefan systems for fluids with arbitrary many components \cite{JuSt14}.
A different entropy density was suggested in \cite{GGJ03} to prove $L^\infty$ 
bounds in one space dimension.
In fact, the idea of using entropy variables already appeared in the analysis 
of parabolic systems from non-equilibrium thermodynamics \cite{DGJ97},
originally used to remove the influence from the electric potential, and
goes back to the use of so-called Slotboom variables in semiconductor modeling
\cite[Section 3.2]{Mar86}.

In this paper, we identify the key elements of this idea and provide
a general global existence result for bounded weak solutions to certain systems. 
Furthermore, the technique is applied to a variety of cross-diffusion systems
derived from a random-walk master equation on a lattice, underlying the
strength and flexibility of the method. 

\subsection{Main results}\label{sec.main}

Our first main result concerns the existence of bounded weak solutions
to \eqref{1.eq}-\eqref{1.bic} under some general structural assumptions. 
Motivated by the comments in Sections \ref{sec.idea} and \ref{sec.illust}, 
we impose the following hypotheses:

\begin{description}
\item[\rm H1] There exists a convex function $h\in C^2(D;[0,\infty))$ ($D\subset\R^n$
open, $n\ge 1$) such that its derivative $Dh:D\to\R^n$ is invertible on $\R^n$.
\item[\rm H2'] Let $D\subset(a,b)^n$ for some $a$, $b\in\R$ with $a<b$ and let
$\alpha_i^*$, $m_i\ge 0$ ($i=1,\ldots,n$) be such that for all $z=(z_1,\ldots,z_n)^\top
\in\R^n$ and $u=(u_1,\ldots,u_n)\in D$,
$$
  z^\top D^2h(u)A(u)z \ge \sum_{i=1}^n\alpha_i(u_i)^2z_i^2,
$$
where either $\alpha_i(u_i)=\alpha_i^*(u_i-a)^{m_i-1}$ or 
$\alpha_i(u_i)=\alpha_i^*(b-u_i)^{m_i-1}$.
\item[\rm H2''] There exists $a^*>0$ such that for all $u\in D$ and $i,j=1,\ldots,n$
for which $m_j>1$, it holds that $|a_{ij}(u)|\le a^*|\alpha_j(u_j)|$.
\item[\rm H3] It holds $A\in C^0(D;\R^{n\times n})$ and there exists $C_f>0$ 
such that for all $u\in D$, $f(u)\cdot Dh(u)\le C_f(1+h(u))$.
\end{description}

Hypothesis H1 shows that the inverse transformation $u=(Dh)^{-1}(w)$ is well defined.
If $w(t)$ is a weak solution to \eqref{1.eqw}, we conclude that 
$u(t)=(Dh)^{-1}(w(t))\in D$, yielding the desired $L^\infty$ bounds on $u(t)$.
Assumption H2', which implies H2, is employed to prove a priori
estimates for $\na u_i^{m_i}$. Hypothesis H2'' is needed to show a bound on 
$\pa_t u_i$.
The condition on $f(u)$ in Hypothesis H3 is needed to derive an a priori
estimate for the solution. For instance, if the entropy density is given by 
\eqref{1.log}, we may choose $f_i(u)=u_i^{\alpha_i}(1-u_1-u_2)^{\beta_i}$
for $\alpha_i$, $\beta_i>0$ and $i=1,2$. It is related to the
quasi-positivity assumption $f_i(u)\ge 0$ for all $u\in(0,\infty)^2$ 
with $u_i=0$ \cite[Section 6]{Bot11}.
In contrast to the structural assumptions of \cite{Le06}, 
Hypotheses H1-H3 are easy to verify as soon as
an entropy density is found (often motivated from the application at hand). 
Under the above conditions, the following theorem holds.

\begin{theorem}[General global existence result]\label{thm.ex}
Let Hypotheses H1, H2', H2'', and H3 hold and let 
$u^0\in L^1(\Omega;\R^n)$ be such that $u^0(x)\in D$ for $x\in\Omega$. 
Then there exists a weak solution $u$ to 
\eqref{1.eq}-\eqref{1.bic} satisfying $u(x,t)\in \overline{D}$ for $x\in\Omega$,
$t>0$ and 
\begin{equation*}
  u\in L^2_{\rm loc}(0,\infty;H^1(\Omega;\R^n)), 
	\quad\pa_t u \in L^2_{\rm loc}(0,\infty;H^1(\Omega;\R^n)').
\end{equation*}
The initial datum is satisfied in the sense of $H^1(\Omega;\R^n)'$.
\end{theorem}

Note that since $D$ is bounded by Hypothesis H2', the theorem
yields global $L^\infty$ bounds on $u$. We remark that we may also 
assume the slightly weaker condition $u^0(x)\in\overline{D}$. 
Indeed, we may approximate $u^0$ by $u^0_\eta(x)\in D$ satisfying $u^0_\eta\to u^0$
a.e., apply the theorem to $u^0_\eta$, and perform the limit $\eta\to 0$. 
We refer to \cite{ChJu13} for details.

The assumptions of the theorem are satisfied for the tumor-growth model and
the Maxwell-Stefan equations -- detailed in Section \ref{sec.vf} --,
including the example of Section \ref{sec.illust}. Furthermore, the 
diffusion matrix
\begin{equation}\label{1.ks}
  A(u) = \begin{pmatrix} 1-u_1 & -u_1 \\ -u_2 & 1-u_2 \end{pmatrix}
\end{equation}
with entropy density \eqref{1.log} and domain \eqref{1.D} also satisfies
Hypotheses H1, H2', and H2''.
The model with this diffusion matrix describes the aggregation of two 
population species with cross-diffusion terms related to the drift term 
of the Keller-Segel system.

For the proof of Theorem \ref{thm.ex}, we first semi-discretize \eqref{1.eqw}
in time with step size $\tau>0$
and regularize this equation by the expression
$\eps(\sum_{|\alpha|=m}(-1)^{m}D^{2\alpha}w+w)$, where $\eps>0$, $D^{2\alpha}$ is
a partial derivative of order $2|\alpha|$, and $m\in\N$ is such that
$H^m(\Omega)\hookrightarrow L^\infty(\Omega)$. This regularization ensures that the
approximate solution $w^{(\tau)}$ is bounded. For the proof of approximate solutions,
we only need Hypotheses H1, H2, and H3. The discrete version of the entropy
identity \eqref{1.dHdt} and Hypothesis H2' yield gradient estimates 
for $u^{(\tau)}=(Dh)^{-1}(w^{(\tau)})$.
Furthermore, uniform estimates on the discrete time derivative of $u^{(\tau)}$
can be shown with the help of Hypothesis H2''. Then, by the discrete
Aubin lemma from \cite{DrJu12}, 
the limit $(\tau,\eps)\to 0$ can be performed yielding the
existence of a weak solution to \eqref{1.eq}-\eqref{1.bic}.

Assumptions H2' and H2'' may be too restrictive in certain applications.
For instance, it may be impossible to find a bounded set $D\subset\R^n$
satisfying H1 or the inequality in H2' is not satisfied
(see the examples in Section \ref{sec.ex}). 
However, we show that our technique can be
adapted to situations in which {\em variants} of Hypotheses H2' and H2'' hold.
To illustrate this idea, we choose a class of cross-diffusion systems modeling the
time evolution of two population species. These models are derived
from a random-walk master equation in the diffusive limit (see Appendix
\ref{sec.deriv}). They generalize several models from the literature
(see Section \ref{sec.ex}). 

We consider two situations. In the first case, we assume that 
volume limitations lead to a limitation of the population densities
(volume-filling case). Then the
densities are nonnegative and bounded by a threshold value which is normalized
to one. This situation occurs, for instance, in the volume-filling Keller-Segel model
\cite{HiPa02} and in ion-channel modeling \cite{GNE02}. 
The diffusion matrix in \eqref{1.eq} reads as (see Appendix \ref{sec.deriv})
\begin{equation}\label{1.Aq}
  A(u) = \begin{pmatrix}
	q(u_3)+u_1 q'(u_3) & u_1q'(u_3) \\
	\beta u_2q'(u_3) & \beta(q(u_3)+u_2q'(u_3))
	\end{pmatrix},
\end{equation}
where $u_1$ and $u_2$ are the densities of the species with bounded total
density, $u_1+u_2\le 1$, and $u_3=1-u_1-u_2$. The function $q$ is related to
the transition probability of a species to move from one cell to a neighboring
cell, and $\beta>0$ is the ratio between the transition rates of both
species. Biologically, $q$ vanishes when the cells are fully packed, i.e.\
if $u_1+u_2=1$, so $q(0)=0$ and $q$ is nondecreasing. 
When only one species is considered, the diffusion equation
corresponds to the equation for the cell density in the volume-filling 
chemotaxis model \cite{Wrz10}.
The special case $q(u_3)=u_3$ was analyzed by Burger et al.\ \cite{BDPS10}. 
We are able to
prove the global existence of bounded weak solutions for $q(u_3)=u_3^s$ for any 
$s\ge 1$ but also more general functions are allowed
(see Theorems \ref{thm.exq1} and \ref{thm.exq2}).

\begin{theorem}[Volume-filling case]\label{thm.exq}
Let $s\ge 1$, $\beta>0$, and $q(y)=y^s$ for $y\ge 0$.
Furthermore, let $u^0=(u_1^0,u_2^0)\in L^1(\Omega;\R^2)$ with 
$u^0_1,u_2^0\ge 0$, $u_1^0+u_2^0\le 1$ in $\Omega$. 
Then there exists a bounded weak solution
$u=(u_1,u_2)$ to \eqref{1.eq}-\eqref{1.bic} with diffusion matrix \eqref{1.Aq}
and $f=0$ satisfying $0\le u_1,u_2\le 1$ and $u_3:=1-u_1-u_2\ge 0$ in $\Omega$, $t>0$, 
$$
  u_i q(u_3)^{1/2},\ q(u_3)^{1/2}\in L^2_{\rm loc}(0,\infty;H^1(\Omega;\R^2)), 
	\quad\pa_t u_i\in L^2_{\rm loc}(0,\infty;H^1(\Omega;\R^2)')
$$
for $i=1,2$, and for all $T>0$ and $\phi=(\phi_1,\phi_2)\in L^2(0,T;H^1(\Omega))^2$,
\begin{align}
  \int_0^T \langle \pa_t u,\phi\rangle dt 
	&+ \sum_{i=1}^2 \beta_i\int_0^T\int_\Omega\big(q(u_3)^{1/2}\na(q(u_3)^{1/2}u_i)
	\nonumber \\
	&{}- 3q(u_3)^{1/2}u_i\na(q(u_3)^{1/2})\big)\cdot\na\phi_i dxdt 
	= 0, \label{1.qw}
\end{align}
where $\beta_1=1$, $\beta_2=\beta$ and $\langle\cdot,\cdot\rangle$ is the
dual product between $H^1(\Omega;\R^2)'$ and $H^1(\Omega;\R^2)$.
The initial datum is satisfied in the sense of $H^1(\Omega;\R^2)'$.
\end{theorem}

We may also consider reaction terms $f(u)$ satisfying a particular structure;
see Remark \ref{rem.f}.
The key idea of the proof is to introduce the entropy density
\begin{equation}\label{1.hq}
  h(u) = u_1(\log u_1-1) + u_2(\log u_2-1) + \int_c^{u_3}\log q(y)dy,
\end{equation}
where $u=(u_1,u_2)\in D=\{(u_1,u_2)\in(0,1)^2:u_1+u_2<1\}$ and $0<c<1$. 
A computation shows that 
Hypotheses H1 and H2 are satisfied but not H2'. Indeed, we will prove 
in Section \ref{sec.thm2} (see \eqref{se.edi}) that
$$
  \na u^\top D^2h(u)A(u)\na u \ge \frac{q(u_3)}{u_1}|\na u_1|^2
	+ \frac{q(u_3)}{u_2}|\na u_2|^2,
$$
and the factors degenerate at $u_1+u_2=1$ since $q(0)=0$.
From this, we are able to conclude a gradient estimate for $u_3$ but not
for $u_1$ or $u_2$. However, this is sufficient to infer the existence
of solutions, employing an extension of Aubin's compactness lemma
(see Appendix \ref{sec.aubin}). 

In the second case, volume-filling effects are not taken into account.
Then the diffusion matrix reads as (see Appendix \ref{sec.deriv})
\begin{equation}\label{1.Ap}
  A(u) = \begin{pmatrix}
	p_1(u)+u_1\pa_1p_1(u) & u_1\pa_2 p_1(u) \\
	u_2\pa_1 p_2(u) & p_2(u)+u_2\pa_2p_2(u)
	\end{pmatrix},
\end{equation}
where $p_1$ and $p_2$ are related to the transition probabilities of the
two species and $\pa_ip_j=\pa p_j/\pa u_i$. 
Note that each row of $A(u)$ is the gradient of a function such that
$$
  \diver(A(u)\na u)_i = \Delta(u_ip_i(u)), \quad i=1,2,
$$
which allows for additional $L^2$ estimates for $u_1$ and $u_2$,
using the duality estimates of Pierre and Schmitt \cite{PiSc97}.
This observation was exploited in \cite{DLM13}.

When the transistion probabilities depend linearly on the densities,
$p_i(u)=\alpha_{i0}+\alpha_{i1}u_1+\alpha_{i2}u_2$ ($i=1,2$), we obtain the well-known
population model of Shigesada, Kawashima, and Teramoto \cite{SKT79} with the
diffusion matrix
\begin{equation}\label{1.skt}
  A(u) = \begin{pmatrix}
	\alpha_{01} + 2\alpha_{11}u_1 + \alpha_{12}u_2 & \alpha_{12}u_1 \\
	\alpha_{21}u_2 & \alpha_{20} + \alpha_{21}u_1 + 2\alpha_{22}u_2
  \end{pmatrix}.
\end{equation}
The maximum principle implies that $u_1$ and $u_2$ are nonnegative.
Less results are known concerning upper bounds. In fact, in one space dimension 
and with coefficients $\alpha_{10}=\alpha_{20}$, Shim \cite{Shi02}
proved uniform upper bounds. Moreover, if cross-diffusion is weaker than
self-diffusion (i.e.\ $\alpha_{12}<\alpha_{22}$, $\alpha_{21}<\alpha_{11}$), 
weak solutions are bounded
and H\"older continuous \cite{Le06}. Upper bounds without any restrictions
are not known so far (at least to our best knowledge).

This model has received a lot of attention in the mathematical literature. 
One of the first existence results is due to Kim \cite{Kim84} who neglected
self-diffusion ($\alpha_{11}=\alpha_{22}=0$) and assumed equal coefficients 
($\alpha_{ij}=1$). The tridiagonal case
$\alpha_{21}=0$ was investigated by Amann \cite{Ama89}, Le \cite{Le02}, and
more recently, by Desvillettes and Trescaces \cite{DeTr13}.
Yagi \cite{Yag93} proved an existence theorem under the assumption
that the diffusion matrix is positive definite ($\alpha_{12}<8\alpha_{11}$,
$\alpha_{21}<8\alpha_{22}$, $\alpha_{12}=\alpha_{21}$). 
The first global existence result without any restriction on the diffusion
coefficients (except positivity) was achieved in \cite{GGJ03} in one space
dimension and in \cite{ChJu04,ChJu06} in several space dimensions.
The case of concave functions $p_1$ and $p_2$, for instance,
\begin{equation}\label{1.pp}
  p_i(u) = \alpha_{i0} + a_{i1}(u_1) + a_{i2}(u_2), \quad\mbox{where }
	a_{i1}(u_1)=\alpha_{i1}u_1^s, \ a_{i2}(u_2)=\alpha_{i2}u_2^s,
\end{equation}
and $i=1,2$, $u=(u_1,u_2)$, $0<s<1$, 
was analyzed by Desvillettes et al.\ \cite{DLM13}. We are
able to generalize this result to the case $1<s<4$ but 
we need to restrict the size of the cross-diffusion coefficients $\alpha_{12}$
and $\alpha_{21}$. More general functions $p_1$ and $p_2$ are possible;
see Section \ref{sec.thm3}. 

\begin{theorem}[No volume-filling case]\label{thm.exp}
Let \eqref{1.pp} hold with $1<s<4$ and $(1-1/s)\alpha_{12}\alpha_{21}
\le\alpha_{11}\alpha_{22}$. 
Furthermore, let $u^0=(u_1^0,u_2^0)\in L^1(\Omega;\R^2)$ with 
$u^0_1,u_2^0\ge 0$ in $\Omega$ and $\int_\Omega h(u^0)dx<\infty$, where 
$h$ is defined in \eqref{1.hp} below. Then there exists a weak solution
$u=(u_1,u_2)$ to \eqref{1.eq}-\eqref{1.bic} with diffusion matrix \eqref{1.Ap}
and $f=0$ satisfying $u_i\ge 0$ in $\Omega$, $t>0$,
\begin{align*}
  & u_i^{s/2},\ u_i^s\in L^2_{\rm loc}(0,\infty;H^1(\Omega;\R^2)), \quad
	u_i\in L^\infty_{\rm loc}(0,\infty;L^s(\Omega;\R^2)), \\
	& \pa_t u_i\in L^1_{\rm loc}(0,\infty;X'), \quad i=1,2,
\end{align*}
where $X=\{\psi\in W^{m,\infty}(\Omega):\na\psi\cdot\nu=0$ on $\pa\Omega\}$, $m>d/2$,
for all $\phi=(\phi_1,\phi_2)\in L^\infty(0,T;X)^2$, 
$$
  \int_0^T\langle\pa_t u,\phi\rangle dt
	+ \sum_{i=1}^2\int_0^T\int_\Omega u_ip_i(u)\Delta\phi dxdt 
	= 0,
$$
and $u(0)=u^0$ in the sense of $X'$.
\end{theorem}

Nonvanishing reaction terms $f(u)$ can be treated if they satisfy
an appropriate growth condition such that $f(u^{(\tau)})$ is bounded
in some $L^p$ space with $p>1$, where $u^{(\tau)}$ is a solution to an 
approximate problem. We leave the details to the reader.

The idea of the proof is to employ the entropy density
\begin{equation}\label{1.hp}
  h(u) = \int_c^{u_1}\int_c^z \frac{a_{21}'(y)}{y}dy dz
	+ \int_c^{u_2}\int_c^z \frac{a_{12}'(y)}{y}dydz
\end{equation}
for $u=(u_1,u_2)\in D=(0,\infty)^2$, where $c>0$. Hypothesis H2 is only satisfied 
if the restriction $(1-1/s)\alpha_{12}\alpha_{21}\le\alpha_{11}\alpha_{22}$ holds. 
It is an open problem whether there exists another entropy density fulfilling 
Hypothesis H2 without any restriction on $\alpha_{ij}$ (except positivity).
In order to satisfy Hypothesis H1, we need to regularize the entropy density
\eqref{1.hp} as in \cite{DLM13}:
$$
  h_\eps(u) = h(u) + \eps u_1(\log u_1-1) + \eps u_2(\log u_2-1).
$$
This regularization is motivated from the population model with diffusion matrix
\eqref{1.skt}, see Section \ref{sec.nvf}.
The range of $Dh_\eps$ equals $\R^2$, as required in Hypothesis H1.
However, this regularization makes necessary to regularize also the diffusion
matrix,
$$
  A_\eps(u) = A(u) + \eps\begin{pmatrix} u_2 & 0 \\ 0 & u_1 \end{pmatrix}.
$$
Then the regularized product $D^2h_\eps(u)A_\eps(u)$ is positive semi-definite 
(Hypothesis H2) only if $s<4$. This restriction may be improved by developing
a better regularization procedure.

This paper is organized as follows. In Section \ref{sec.ex}, we consider some
examples of cross-diffusion systems, studied in the literature, and
discuss the validity of our hypotheses.
The general existence Theorem \ref{thm.ex} is shown in Section \ref{sec.thm1}.
The proofs of Theorems \ref{thm.exq} and \ref{thm.exp} are presented in
Sections \ref{sec.thm2} and \ref{sec.thm3}, respectively. 
Some further results and open problems are discussed in Section \ref{sec.open}. 
The appendix is concerned with some relations of our method to non-equilibrium
thermodynamics; the derivation of a general population diffusion model,
containing the diffusion matrices \eqref{1.Aq} and \eqref{1.Ap} as special cases;
and the proof of a variant of the Aubin compactness lemma, needed in the proof
of Theorem \ref{thm.exq}. To simplify the presentation, we write in the following
$L^p(0,T;H^k(\Omega))$ instead of $L^p(0,T;H^k(\Omega;\R^n))$.


\section{Examples}\label{sec.ex}

We present some diffusion problems from physical and biological applications
and discuss the validity of Hypotheses H1-H3. Some examples do not satisfy the
hypotheses but similar ideas as detailed in the introduction apply.

\subsection{Examples with volume filling}\label{sec.vf}

\subsubsection*{Volume-filling model of Burger.}\label{sec.burger}
Burger et al.\ \cite{BSW12} derived from a lattice-based hopping model
a cross-diffusion system for $n$ species, incorporating volume-filling effects
and modeling the ion transport through narrow pores. The system consists
of equation \eqref{1.eq} with the diffusion matrix $A(u)=(a_{ij}(u))$, where
$$
  a_{ij}(u) = D_iu_i\quad\mbox{for }i\neq j, \quad 
	a_{ii}(u) = D_i(1-\rho+u_i), \quad\rho=\sum_{j=1}^n u_j,
$$
and $D_i>0$ are some constants.
It was shown in \cite{BSW12} that the entropy with density
$$
  h(u) = \sum_{i=1}^n u_i(\log u_i-1) + (1-\rho)(\log(1-\rho)-1),
$$
where $u=(u_1,\ldots,u_n)\in D=\{u\in(0,\infty)^n:\sum_{i=1}^n u_i<1\}$,
is a Lyapunov functional along solutions to \eqref{1.eq} with $f=0$.
This entropy density satisfies Hypothesis H1 since $u_i=((Dh)^{-1}(w))_i
= e^{w_i}/(1+\sum_{j=1}^n e^{w_j})\in D$ for $w=(w_1,\ldots,w_n)\in\R^n$.
The case of $n=2$ species was investigated in \cite{BDPS10}.
Then the diffusion matrix becomes
$$
  A(u) = \begin{pmatrix}
	D_1(1-u_2) & D_1u_1 \\ D_2u_2 & D_2(1-u_1)
	\end{pmatrix},
$$
which is a special case of \eqref{1.Aq} with $q(u_3)=u_3$ and $\beta=D_2/D_1$.
A computation shows that for $z=(z_1,z_2)\in\R^2$,
\begin{align*}
  z^\top D^2h(u)A(u)z
	&= D_1(1-\rho)\left(\frac{z_1^2}{u_1}+\frac{z_2^2}{u_2}\right)
	+ D_1\frac{2-\rho}{1-\rho}(z_1+z_2)^2 \\
	&\phantom{xx}{}+ (D_2-D_1)\frac{u_2}{1-\rho}
	\left(z_1+\frac{1-u_1}{u_2}z_2\right)^2.
\end{align*}
Consequently, assuming without loss of generality that $D_2>D_1$ (otherwise,
change the indices 1 and 2),
Hypothesis H2 is satisfied but not Hypothesis H2' since the
quadratic form degenerates at $\rho=1$. Burger et al.\ have shown in \cite{BDPS10}
that a global existence analysis is still possible. We generalize this result
to general functions $q$ in Theorem \ref{thm.exq}.

In the general case $n>2$ and the case of equal diffusivity constants
$D_1=D_i$ for all $i=2,\ldots,n$, summing up all equations yields
$\pa_t\rho-\Delta\rho=0$ in $\Omega$ with homogeneous Neumann boundary conditions.
Thus, by the classical maximum principle, $\rho$ is bounded from below
and above, and this yields $L^\infty$ bounds for the components $u_i\ge 0$.
The analysis of different diffusivity constants and $n>2$ is much more delicate
and an open problem.

\subsubsection*{Tumor-growth model}\label{sec.tumor}
Jackson and Byrne \cite{JaBy02} derived a continuous mechanical model 
for the growth of symmetric avascular tumors in one space dimension. 
In this model, the mass balance
equations for the volume fractions of the tumor cells $u_1$, the extracellular 
matrix $u_2$, and the water phases $u_3=1-u_1-u_2$ are supplemented by equations
for the velocity, obtained from a force balance. For a small cell-induced pressure
coefficient (i.e.\ $\theta=\beta=1$ in \cite{JaBy02}), 
the resulting diffusion matrix reads as
$$
  A(u) = \begin{pmatrix}
	u_1(1-u_1)-u_1u_2^2  & -u_1u_2(1+u_1) \\
	-u_1u_2+u_2^2(1-u_2) & u_2(1-u_2)(1+u_1)
	\end{pmatrix}.
$$
With the entropy density \eqref{1.log}, defined in $D$ which is given by \eqref{1.D},
we compute for $z=(z_1,z_2)^\top\in\R^2$,
$$
  z^\top D^2h(u)A(u)z = z_1^2 + (1+u_1)z_2^2 + u_1z_1z_2
  \ge \frac12 z_1^2 + \left(1+u_1-\frac12 u_2^2\right)z_2^2.
$$
Thus, since $u_2\le 1$, 
Hypotheses H1, H2', and H2'' are satisfied and the global existence
result follows from Theorem \ref{thm.exq}
if the reaction terms satisfy Hypothesis H3. This result was first proven in 
\cite{JuSt12}.

\subsubsection*{Maxwell-Stefan equations}\label{sec.ms}

The Maxwell-Stefan equations describe the diffusive transport of multicomponent
gaseous mixtures. In the case of mixtures of ideal gases consisting of $n$ components
under isobaric, isothermal conditions with vanishing barycentric velocity, 
the equations read as
$$
  \pa_t u_i + \diver J_i = f_i(u), \quad 
	\na u_i = -\sum_{j\neq i}\frac{u_jJ_i-u_iJ_j}{D_{ij}}, \quad i=1,\ldots,n,
$$
where $D_{ij}>0$ are the binary diffusion coefficients. The variables $u_i$
are the molar concentrations of the mixture and they satisfy $\sum_{j=1}^n u_j=1$.
The inversion of the flux-gradient relations is not straightforward since the
linear system in $J_i$ has a singular matrix; see \cite{Bot11,MaTe13}.
The idea of \cite{JuSt14} was to replace the last component $u_n$ by the remaining
ones by $u_n=1-\sum_{j=1}^{n-1}u_i$ and to analyze the remaining $n-1$ equations.
For $n=3$ components, the inversion can be made explicit, leading to
system \eqref{1.eq} for the remaining $n-1=2$ equations with the diffusion matrix
\begin{equation}\label{ex.ms}
  A(u) = \frac{1}{\delta(u)}
	\begin{pmatrix}
	d_2 + (d_0-d_2)u_1 & (d_0-d_1)u_1 \\
	(d_0-d_2)u_2 & d_1 + (d_0-d_1)u_2
	\end{pmatrix},
\end{equation}
where we abbreviated $d_{i+j-2}=D_{ij}$ and 
$\delta(u)=d_1d_2(1-u_1-u_2)+d_0(d_1u_1+d_2u_2)$.
The diffusion matrix \eqref{1.ms} in Section \ref{sec.illust}
is obtained after setting $d_0=3$, $d_1=2$, and $d_2=1$.

Employing the entropy density \eqref{1.log} as in the previous example, we compute
$$
  z^\top D^2h(u)A(u)z = \frac{d_2}{u_1}z_1^2 + \frac{d_1}{u_2}z_2^2
	+ \frac{d_0 u_1u_2}{1-u_1-u_2}(z_1+z_2)^2
$$
for $z=(z_1,z_2)^\top\in\R^2$. Therefore, Hypothesis H1 and H2 are fulfilled
with $D$ given by \eqref{1.D}. Moreover, 
Hypothesis H2' is satisfied with $m_1=m_2=0$, and also Hypothesis H2'' holds. 
Theorem \ref{thm.ex} yields the existence of
global bounded weak solutions to this problem. It can be shown that this
result holds true in the case of $n>3$ components; see \cite{JuSt14}.


\subsection{Models without volume filling}\label{sec.nvf}

\subsubsection*{Population model of Shigesada, Kawashima, and Teramoto}\label{sec.skt}
The model describes the evolution of two population densities $u_1$ and $u_2$
governed by equation \eqref{1.eq} with the diffusion matrix
$$
  A(u) = \begin{pmatrix}
	\alpha_{01} + 2\alpha_{11}u_1 + \alpha_{12}u_2 & \alpha_{12}u_1 \\
	\alpha_{21}u_2 & \alpha_{20} + \alpha_{21}u_1 + 2\alpha_{22}u_2
  \end{pmatrix},
$$
where the coefficients $\alpha_{ij}$ are nonnegative.
Introducing the entropy density
$$
  h(u) = \frac{u_1}{\alpha_{12}}(\log u_1-1) + \frac{u_2}{\alpha_{21}}(\log u_2-1), 
	\quad u=(u_1,u_2)\in D=(0,\infty)^2,
$$
a computation shows that Hypothesis H2 is satisfied, i.e.\ for all 
$z=(z_1,z_2)^\top\in\R^2$,
$$
  z^\top D^2h(u)A(u)z = \left(\frac{\alpha_{11}}{\alpha_{12}}
	+\frac{\alpha_{10}}{\alpha_{12}u_1}\right)z_1^2
	+ \left(\frac{\alpha_{22}}{\alpha_{21}}+\frac{\alpha_{20}}{\alpha_{21}u_2}\right)z_2^2
  + \left(\sqrt{\frac{u_2}{u_1}}z_1+\sqrt{\frac{u_1}{u_2}}z_2\right)^2.
$$
Thus, we see that Hypothesis H2' 
with $m_i=0$ or $m_i=1$ and Hypothesis H2'' hold true. Since 
$D=(0,\infty)^2$ is not bounded and the range of $Dh$ is not the whole $\R^2$,
we cannot apply Theorem \ref{thm.ex}. 
However, as the coefficients of $A(u)$ depend only linearly on $u$, the proof
of Theorem \ref{thm.exp} can be adapted to this case and we conclude 
the existence of global nonnegative weak solutions to this problem.
This statement was first proved in \cite{ChJu04} 
using a different approximation procedure.
In Theorem \ref{thm.exp} we generalize this result.

\subsubsection*{Semiconductor model with electron-hole scattering}\label{sec.semi}

The carrier transport through a semiconductor device with strong
electron-hole scattering but vanishing electric field 
can be modeled by equation \eqref{1.eq} with diffusion matrix
$$
  A(u) = \frac{1}{1+\mu_2u_1+\mu_1u_2}\begin{pmatrix}
	\mu_1(1+\mu_2 u_1) & \mu_1\mu_2 u_1 \\
	\mu_1\mu_2 u_2     & \mu_2(1+\mu_1 u_2)
	\end{pmatrix},
$$
where $u_1$ and $u_2$ are the electron and hole densities, respectively,
and $\mu_1$ and $\mu_2$ denote the corresponding (positive) mobility constants.
This model was formally derived from the semiconductor Boltzmann equation 
with a collision operator taking into account strong electron-hole scattering 
\cite{Rez95}. The global existence of weak solutions was shown in \cite{ChJu07}. 

Introducing the entropy density
$$
  h(u) = u_1(\log u_1-1)+u_2(\log u_2-1), 
$$
where $u=(u_1,u_2)\in D=(0,\infty)^2$, we obtain for $z=(z_1,z_2)^\top\in\R^2$,
$$
  z^\top D^2h(u)A(u)z = \frac{1}{1+\mu_2u_1+\mu_1u_2}\left(
	\frac{\mu_1}{u_1}z_1^2 + \frac{\mu_2}{u_2}z_2^2
	+ \mu_1\mu_2(z_1+z_2)^2\right).
$$
Thus, Hypothesis H2 is satisfied but not H2' since the quadratic form is
not uniformly positive. However, as shown in \cite{ChJu07}, bounds on the 
electron and hole masses (i.e.\ the integrals of $u_1$ and $u_2$)
together with estimates from the entropy dissipation yield an $H^1$
bound for $u_1^{1/2}$ and $u_2^{1/2}$. Then, with the entropy variables
$w_i=\pa h/\pa u_i=\log u_i$, 
the existence of global weak solutions was proved in \cite{ChJu07}.


\section{Proof of Theorem \ref{thm.ex}}\label{sec.thm1}

The proof is based on the solution
of a time-discrete and regularized problem, for which only Hypotheses 
H1, H2, and H3 are needed. 

{\em Step 1. Solution of an approximate problem.}
Let $T>0$, $N\in\N$, $\tau=T/N$, $m\in\N$ with $m>d/2$, and $\tau>0$.
Let $w^{k-1}\in L^\infty(\Omega;\R^n)$ be given. (If $k=1$, we define
$w^0=(Dh)^{-1}(u^0)$ which is possible since we assumed that $u^0(x)\in D$ for
$x\in\Omega$.)
We wish to find $w^k\in H^{m}(\Omega;\R^n)$ such that 
\begin{align}
  \frac{1}{\tau}\int_\Omega & (u(w^k)-u(w^{k-1}))\cdot\phi dx
	+ \int_\Omega\na\phi:B(w^k)\na w^k dx \nonumber \\
	&{}+ \eps\int_\Omega\bigg(\sum_{|\alpha|=m}D^{\alpha}w^k\cdot D^{\alpha}\phi 
	+ w^k\cdot\phi\bigg)dx
  = \int_\Omega f(u(w^k))\cdot\phi dx \label{a.reg}
\end{align}
for all $\phi\in H^m(\Omega;\R^n)$,
where $\alpha=(\alpha_1,\ldots,\alpha_n)\in\N_0^n$ with 
$|\alpha|=\alpha_1+\cdots+\alpha_n=m$ is a multiindex and 
$D^\alpha=\pa^m/(\pa x_1^{\alpha_1}\cdots\pa x_n^{\alpha_n})$ 
is a partial derivative of order $m$, and
$u(w)=(Dh)^{-1}(w)$. We observe that $m$ is chosen in such a way that
$H^{m}(\Omega)\hookrightarrow L^\infty(\Omega)$. 

\begin{lemma}[Existence for the regularized system \eqref{a.reg}]\label{lem.ex}
Let Assumptions H1, H2, and H3 hold. Let $u^0\in L^\infty(\Omega;\R^n)$ be such that
$u^0(x)\in D$ for $x\in\Omega$, and let $0<\tau<1/C_f$, where $C_f>0$ is defined in H3.
Then there exists a weak solution $w^k\in H^m(\Omega;\R^n)$ to
\eqref{a.reg} satisyfing $u(w^k(x))\in D$ for $x\in\Omega$ and the 
discrete entropy-dissipation inequality
\begin{align}
  (1-C_f\tau)\int_\Omega h(u(w^k))dx &+ \tau\int_\Omega\na w^k:B(w^k)\na w^k dx
	+ \eps\tau\int_\Omega\bigg(\sum_{|\alpha|=m}|D^{\alpha}w|^2 + |w|^2\bigg)dx 
	\nonumber \\
	&\le \tau C_f\mbox{\rm meas}(\Omega) + \int_\Omega h(u(w^{k-1}))dx. \label{a.edi}
\end{align}
\end{lemma}

\begin{proof}
Let $y\in L^\infty(\Omega;\R^n)$ and $\delta\in[0,1]$ be given.
We solve first the following linear problem:
\begin{equation}\label{a.a}
  a(w,\phi) = F(\phi)\quad\mbox{for all }\phi\in H^{m}(\Omega;\R^n),
\end{equation}
where
\begin{align*}
  a(w,\phi) &= \int_\Omega\na\phi:B(y)\na w dx
	+ \eps\int_\Omega\bigg(\sum_{|\alpha|=m}D^{\alpha}w\cdot D^{\alpha}\phi 
	+ w\cdot\phi\bigg)dx, \\
	F(\phi) &= -\frac{\delta}{\tau}\int_\Omega\big(u(y)-u(w^{k-1})\big)\cdot\phi dx
	+ \delta\int_\Omega f(u(y))\cdot\phi dx.
\end{align*} 
The forms $a$ and $F$ are bounded on $H^{m}(\Omega;\R^n)$. 
The matrix $B(y)=A(u(y))(D^2h)^{-1}(u(y))$ 
is positive semi-definite since, by Assumption H2,
$$
  z^\top B(y)z = ((D^2h)^{-1}z)^\top(D^2h)A(u(y))((D^2h)^{-1}z) \ge 0
	\quad\mbox{for all }z\in\R^n.
$$
Hence, the bilinear form $a$ is coercive:
$$
  a(w,w) \ge \eps\int_\Omega\bigg(\sum_{|\alpha|=m}|D^\alpha w|^2 + |w|^2\bigg)dx 
	\ge \eps C\|w\|^2_{H^{m}(\Omega)} \quad\mbox{for }
	w\in H^m(\Omega;\R^n).
$$
The last inequality follows from the generalized Poincar\'e inequality for
$H^m$ spaces \cite[Chap.~II.1.4, Formula (1.39)]{Tem97}, and $C>0$ is some
constant only depending on $\Omega$.
Therefore, we can apply the Lax-Milgram lemma to obtain the existence of a unique
solution $w\in H^{m}(\Omega;\R^n)\hookrightarrow L^\infty(\Omega;\R^n)$ 
to \eqref{a.a}. This defines the fixed-point operator
$S:L^\infty(\Omega;\R^n)\times[0,1]\to L^\infty(\Omega;\R^n)$, $S(y,\delta)=w$, where
$w$ solves \eqref{a.a}. 

It holds $S(y,0)=0$ for all $y\in L^\infty(\Omega;\R^n)$. We claim that
the operator $S$ is continuous. Indeed, let $\delta_\eta\to\delta$ in $\R$ and
$y_\eta\to y$ strongly in $L^\infty(\Omega;\R^n)$ as $\eta\to 0$ and set
$w_\eta=S(y_\eta,\delta_\eta)\in H^{m}(\Omega;\R^n)$. Then, by continuity, 
$u(y_\eta)\to u(y)$, $B(y_\eta)\to B(y)$, and $f(u(y_\eta))\to f(u(y))$ 
strongly in $L^\infty(\Omega;\R^n)$, and the above
coercivity estimate gives a uniform bound for $(w_\eta)$ in $H^{m}(\Omega;\R^n)$.
This implies that, for a subsequence which is not relabeled,
$w_\eta\rightharpoonup w$ weakly in $H^{m}(\Omega;\R^n)$. Then, performing the
limit $\eta\to 0$ in \eqref{a.a}, it follows that $w=S(y,\delta)$. In view
of the compact embedding $H^{m}(\Omega;\R^n)\hookrightarrow L^\infty(\Omega;\R^n)$,
we infer that for a subsequence, $S(y_\eta,\delta_\eta)=w_\eta\to w=S(y,\delta)$
strongly in $L^\infty(\Omega;\R^n)$. Since the limit $w$ is unique, the whole
sequence $(w_\eta)$ is converging which shows the continuity. Furthermore,
the compact embedding $H^{m}(\Omega;\R^n)\hookrightarrow L^\infty(\Omega;\R^n)$
yields the compactness of $S$.

It remains to prove a uniform bound for all fixed points of $S(\cdot,\delta)$ in
$L^\infty(\Omega;\R^n)$. Let $w\in L^\infty(\Omega;\R^n)$ be such a fixed point.
Then $w$ solves \eqref{a.a} with $y$ replaced by $w$. 
With the test function $\phi=w$, we find that
\begin{align}
  \frac{\delta}{\tau}\int_\Omega (u(w)-u(w^{k-1}))\cdot w dx
	&+ \int_\Omega \na w:B(w)\na w dx \nonumber \\
	&+ \eps\int_\Omega\bigg(\sum_{|\alpha|=m}|D^\alpha w|^2 + |w|^2\bigg)dx 
	= \delta\int_\Omega f(u(w))\cdot w dx. \label{a.aux}
\end{align}
The convexity of $h$ implies that $h(x)-h(y)\le Dh(x)\cdot(x-y)$ for
all $x$, $y\in D$. Choosing $x=u(w)$ and $y=u(w^{k-1})$ and using
$Dh(u(w))=w$, this gives
$$
  \frac{\delta}{\tau}\int_\Omega (u(w)-u(w^{k-1}))\cdot w dx
	\ge \frac{\delta}{\tau}\int_\Omega\big(h(u(w))-h(u(w^{k-1}))\big)dx.
$$
Taking into account the positive semi-definiteness of $B(w)$ 
and Assumption H3, \eqref{a.aux} can be estimated as follows:
\begin{align*}
  \delta\int_\Omega h(u(w))dx 
	&+ \eps\tau\int_\Omega\bigg(\sum_{|\alpha|=m}|D^\alpha w|^2 + |w|^2\bigg)dx \\
	&\le C_f\tau\delta \int_\Omega(1+h(u(w)))dx	+ \delta\int_\Omega h(u(w^{k-1}))dx.
\end{align*}
Choosing $\tau<1/C_f$, this yields an $H^{m}$ bound for $w$ uniform in 
$\delta$ (not uniform in $\tau$ or $\eps$).
The Leray-Schauder fixed-point theorem shows that there exists a solution
$w\in H^{m}(\Omega;\R^n)$ to \eqref{a.a} with $y$ replaced by $w$ and $\delta=1$.
\end{proof}

{\em Step 2. Uniform bounds.} 
By Lemma \ref{lem.ex}, there exists a weak solution $w^k\in H^{m}(\Omega;\R^n)$
to \eqref{a.reg}. Because of the boundedness of $D$, $a<u_i(w^k)<b$ for
$i=1,\ldots,n$. We need a priori estimates uniform in $\tau$
and $\eps$. For this, let $w^{(\tau)}(x,t)=w^k(x)$ and
$u^{(\tau)}(x,t)=u(w^k(x))$ for $x\in\Omega$ and
$t\in((k-1)\tau,k\tau]$, $k=1,\ldots,N$. At time $t=0$, we set 
$w^{(\tau)}(\cdot,0)=Dh(u^0)$ and 
$u^{(\tau)}(\cdot,0)=u^0$. Let $u^{(\tau)}=(u_1^{(\tau)},\ldots,u_n^{(\tau)})$.
Furthermore, we introduce the shift operator
$(\sigma_\tau u^{(\tau)})(\cdot,t)=u(w^{k-1}(x))$ for $x\in\Omega$,
$(k-1)\tau<t\le k\tau$, $k=1,\ldots,N$.
Then $u^{(\tau)}$ solves the equation
\begin{align}
  \frac{1}{\tau}\int_0^T & \int_\Omega(u^{(\tau)}-\sigma_\tau u^{(\tau)})\cdot\phi dxdt
	+ \int_0^T\int_\Omega\na\phi:B(w^{(\tau)})\na w^{(\tau)}dxdt \nonumber \\
	&{}+ \eps\int_0^T\int_\Omega\bigg(\sum_{|\alpha|=m}D^\alpha w^{(\tau)}\cdot
	D^\alpha\phi + w^{(\tau)}\cdot\phi\bigg)dxdt
	= \int_0^T\int_\Omega f(u^{(\tau)})\cdot\phi dxdt \label{a.aux2}
\end{align}
for piecewise constant functions $\phi:(0,T)\to H^m(\Omega;\R^n)$. We note that
this set of functions is dense in $L^2(0,T;H^m(\Omega;\R^n))$ 
\cite[Prop. 1.36]{Rou05}, such that the weak formulation also 
holds for such functions. By Assumption H2' and $\na w^{(\tau)}
=D^2h(u^{(\tau)})\na u^{(\tau)}$, we find that
\begin{align*}
  \int_\Omega \na w^{(\tau)}:B(w^{(\tau)})\na w^{(\tau)} dx 
	&= \int_\Omega \na u^{(\tau)}:D^2 h(u^{(\tau)})A(u^{(\tau)})\na u^{(\tau)} dx \\
	&\ge \sum_{i=1}^n\int_\Omega \alpha_i(u_i^{(\tau)})^2|\na u_i^{(\tau)}|^2 dx
	= \sum_{i=1}^n\int_\Omega |\na\widetilde\alpha_i(u_i^{(\tau)})|^2 dx,
\end{align*}
where $\widetilde\alpha'_i=\alpha_i$. Then either
$\widetilde\alpha_i(y)=(\alpha_i^*/m_i)(y-a)^{m_i}$ or 
$\widetilde\alpha_i(y)=(\alpha_i^*/m_i)(b-y)^{m_i}$. It follows from
\eqref{a.edi} that
\begin{align*}
  (1- & C_f\tau)\int_\Omega h(u(w^k))dx 
	+ \tau\sum_{i=1}^n\int_\Omega|\na\widetilde\alpha_i(u_i(w^k))|^2 dx \\
	&{}+ \eps\tau\int_\Omega\bigg(\sum_{|\alpha|=m}
	|D^\alpha w^k|^2 + |w^k|^2\bigg)dx 
  \le C_f\tau\mbox{meas}(\Omega) + \int_\Omega h(u(w^{k-1}))dx.
\end{align*}
Solving these inequalities recursively, we infer that
\begin{align*}
  \int_\Omega & h(u(w^k))dx 
	+ \tau\sum_{j=1}^k(1-C_f\tau)^{-(k+1-j)}\sum_{i=1}^n
	\int_\Omega|\na\widetilde\alpha_i(u_i(w^j))|^2 dx \\
	&\phantom{xx}{}
	+ \eps\tau\sum_{j=1}^k(1-C_f\tau)^{-(k+1-j)}\int_\Omega\bigg(\sum_{|\alpha|=m}
	|D^\alpha w^j|^2 + |w^j|^2\bigg)dx \\
  &\le C_f\tau\mbox{meas}(\Omega)\sum_{j=1}^k(1-C_f\tau)^{-j} + \int_\Omega h(u^0)dx.  
\end{align*}
Since $(1-C_f\tau)^{-(k+1-j)}\ge 1$ for $j=1,\ldots,k$
and $\sum_{j=1}^k(1-C_f\tau)^{-j}\le (\exp(C_f k\tau)+1)/(C_f\tau)$, we obtain
for $k\tau\le T$,
\begin{align*}
  \int_\Omega h(u(w^k))dx 
	&+ \tau\sum_{j=1}^k\sum_{i=1}^n\int_\Omega|\na\widetilde\alpha_i(u_i(w^j))|^2 dx
	+ \eps\tau\sum_{j=1}^k\int_\Omega\bigg(\sum_{|\alpha|=m}
	|D^\alpha w^j|^2 + |w^j|^2\bigg)dx \\
  &\le \mbox{meas}(\Omega)(1+e^{C_fT}) + \int_\Omega h(u^0)dx.  
\end{align*}
Together with the $L^\infty$ bounds for $u^{(\tau)}$, this gives the
following uniform bounds:
\begin{equation}\label{a.entest}
  \|\widetilde\alpha(u^{(\tau)})\|_{L^2(0,T;H^1(\Omega))} \le C, \quad
	\sqrt{\eps}\|w^{(\tau)}\|_{L^2(0,T;H^m(\Omega))} \le C,
\end{equation}
where $C>0$ denotes here and in the following a generic constant independent
of $\tau$ and $\eps$. For $m_i\le 1$ (appearing in $\widetilde\alpha_i$), we have
\begin{align}
  \|\na u_i^{(\tau)}\|_{L^2(0,T;L^2(\Omega))}
	&= \|\alpha_i(u_i^{(\tau)})^{-1}\na\widetilde\alpha_i(u_i^{(\tau)})
	\|_{L^2(0,T;L^2(\Omega))} \nonumber \\
	&\le (\alpha_i^*)^{-1}(b-a)^{1-m_i}
	\|\na\widetilde\alpha_i(u_i^{(\tau)})\|_{L^2(0,T;L^2(\Omega))} \le C.
	\label{a.u1}
\end{align}
If $m_i>1$, it follows that either
\begin{equation}\label{a.u2}
  \begin{aligned}
  \|\na (u_i^{(\tau)}-a)^{m_i}\|_{L^2(0,T;L^2(\Omega))}
	&= (\alpha_i^*)^{-1}\|\na\widetilde\alpha_i(u_i^{(\tau)})\|_{L^2(0,T;L^2(\Omega))} 
	\le C\quad\mbox{or} \\
	\|\na (b-u_i^{(\tau)})^{m_i}\|_{L^2(0,T;L^2(\Omega))}
	&= (\alpha_i^*)^{-1}\|\na\widetilde\alpha_i(u_i^{(\tau)})\|_{L^2(0,T;L^2(\Omega))} 
	\le C.
	\end{aligned}
\end{equation}

In order to derive a uniform estimate for the discrete time derivative,
let $\phi\in L^2(0,T;$ $H^m(\Omega;\R^n))$. Then
\begin{align*}
  \frac{1}{\tau} & \left|\int_\tau^T\int_\Omega(u^{(\tau)}-\sigma_\tau u^{(\tau)})
	\cdot \phi dxdt\right| \\
	&\le \|A(u^{(\tau)})\na u^{(\tau)}\|_{L^2(0,T;L^2(\Omega))}
	\|\na\phi\|_{L^2(0,T;L^2(\Omega))} \\
	&\phantom{xx}{}+ \eps\|w^{(\tau)}\|_{L^2(0,T;H^m(\Omega))}
	\|\phi\|_{L^2(0,T;H^m(\Omega))}
	+ \|f(u^{(\tau)})\|_{L^2(0,T;L^2(\Omega))}\|\phi\|_{L^2(0,T;L^2(\Omega))}.
\end{align*}
The first term on the right-hand side is uniformly bounded since, by 
Assumption H2'', 
\begin{align*}
  \|(A(u^{(\tau)})\na u^{(\tau)})_i\|_{L^2(0,T;L^2(\Omega))}^2
	&\le \sum_{j=1,\ m_j>1}^n\left\|\frac{a_{ij}(u^{(\tau)})}{\alpha_j(u_j^{(\tau)})}
	\right\|_{L^\infty(0,T;L^\infty(\Omega))}^2
	\|\na\widetilde\alpha_j(u_j^{(\tau)})\|_{L^2(0,T;L^2(\Omega))}^2 \\
	&\phantom{xx}{}+
	\sum_{j=1,\ m_j\le 1}^n\|a_{ij}(u^{(\tau)})\|_{L^\infty(0,T;L^\infty(\Omega))}^2
	\|\na u_j^{(\tau)}\|_{L^2(0,T;L^2(\Omega))}^2 \\
	&\le C,
\end{align*}
using \eqref{a.entest} and \eqref{a.u1}.
Thus, by the $L^\infty$ bound for $u^{(\tau)}$ and \eqref{a.entest},
\begin{equation}\label{a.uteps}
  \frac{1}{\tau} \left|\int_\tau^T\int_\Omega(u^{(\tau)}-\sigma_\tau u^{(\tau)})
	\cdot \phi dxdt\right|
	\le C\big(\sqrt{\eps}\|\phi\|_{L^2(0,T;H^m(\Omega))}
	+ \|\phi\|_{L^2(0,T;H^1(\Omega))}\big),
\end{equation}
which shows that
\begin{equation}\label{a.ut}
  \tau^{-1}\|u^{(\tau)}-\sigma_\tau u^{(\tau)}\|_{L^2(\tau,T;(H^m(\Omega))')} \le C.
\end{equation}

{\em Step 3. The limit $(\tau,\eps)\to 0$.}
The uniform estimates \eqref{a.ut} and either \eqref{a.u1} or \eqref{a.u2}
allow us to apply the Aubin lemma in the version of \cite[Theorem 1]{DrJu12} 
(if $m_i\le 1$) or in the version of \cite[Theorem 3a]{CJL13} (if $m_i>1$;
apply the theorem to $v_i^{(\tau)}=u_i^{(\tau)}-a$ or $v_i^{(\tau)}=b-u_i^{(\tau)}$),
showing that, up to a subsequence which is not relabeled, as $(\tau,\eps)\to 0$,
$$
  u^{(\tau)} \to u \quad\mbox{strongly in }L^2(0,T;L^2(\Omega))\mbox{ and a.e. in }
	\Omega\times(0,T).
$$
Because of the boundedness of $u_i^{(\tau)}$ in $L^\infty$, this convergence
even holds in $L^p(0,T;L^p(\Omega))$ for any $p<\infty$, which is a consequence
of the dominated convergence theorem. In particular,
$A(u^{(\tau)})\to A(u)$ strongly in $L^p(0,T;L^p(\Omega))$.
Furthermore, by weak compactness, for a subsequence,
\begin{align*}
  A(u^{(\tau)})\na u^{(\tau)} \rightharpoonup U 
	&\quad\mbox{weakly in }L^2(0,T;L^2(\Omega)), \\
	\tau^{-1}(u^{(\tau)}-\sigma_\tau u^{(\tau)}) \rightharpoonup \pa_t u
	&\quad\mbox{weakly in }L^2(0,T;H^m(\Omega)'), \\
	\eps w^{(\tau)}\to 0 &\quad\mbox{strongly in }L^2(0,T;H^m(\Omega)).
\end{align*}

We claim that $U=A(u)\na u$. Indeed, observing that
$\widetilde\alpha_i(u_i^{(\tau)})\to \widetilde\alpha_i(u_i)$ a.e., it follows that
$$
  \na\widetilde\alpha_i(u_i^{(\tau)})\rightharpoonup \na\widetilde\alpha_i(u_i)
	\quad\mbox{weakly in }L^2(0,T;L^2(\Omega)).
$$
Let $m_j>1$.
Since the quotient $a_{ij}(u^{(\tau)})/\alpha_j(u_j^{(\tau)})$ is bounded,
by Assumption H2'', a subsequence converges weakly* in $L^\infty$. Taking into
account the a.e.\ convergence of $(u^{(\tau)})$, this implies that 
$a_{ij}(u^{(\tau)})/\alpha_j(u_j^{(\tau)})\rightharpoonup^* a_{ij}(u)/\alpha_j(u_j)$
weakly* in $L^\infty(0,T;L^\infty(\Omega))$ and a.e.\ and hence strongly in any
$L^p(0,T;L^p(\Omega))$ for $p<\infty$. 
Furthermore, if $m_j\le 1$,
since $a_{ij}$ is continuous in $\overline{D}$ and $u^{(\tau)}(x,t)\in D$,
$a_{ij}(u^{(\tau)})\to a_{ij}(u)$ strongly in $L^p(0,T;L^p(\Omega))$
for all $p<\infty$. We conclude that
\begin{align*}
  (A(u^{(\tau)})\na u^{(\tau)})_i
	&= \sum_{j=1,\ m_j>1}^n \frac{a_{ij}(u^{(\tau)})}{\alpha_j(u_j^{(\tau)})}
	\na\widetilde\alpha_j(u_j^{(\tau)})
	+ \sum_{j=1,\ m_j\le 1}^n a_{ij}(u^{(\tau)})\na u_j^{(\tau)} \\
	&\rightharpoonup \sum_{j=1\ m_j\le 1}^n \frac{a_{ij}(u)}{\alpha_j(u_j)}
	\na\widetilde\alpha_j(u_j) + \sum_{j=1,\ m_j\le 1}^n a_{ij}(u)\na u_j
	= (A(u)\na u)_i
\end{align*}
weakly in $L^q(0,T;L^q(\Omega))$ for all $q<2$. Thus proves the claim.
Finally, we observe that $f(u^{(\tau)})\to f(u)$ strongly in $L^p(0,T;L^p(\Omega))$ 
for $p<\infty$. Therefore, we can pass to the limit $(\tau,\eps)\to 0$ in 
\eqref{a.aux2} to obtain a solution to 
$$
  \int_0^T\langle\pa_t u,\phi\rangle dt
	+ \int_0^T\int_\Omega \na\phi:A(u)\na u dxdt 
	= \int_0^T\int_\Omega f(u)\cdot\phi dxdt
$$
for all $\phi\in L^2(0,T;H^m(\Omega;\R^n))$. In fact, 
performing the limit $\eps\to 0$ and then $\tau\to 0$, we see from \eqref{a.uteps}
that $\pa_t u\in L^2(0,T;H^1(\Omega)')$, and consequently, the 
weak formulation also holds for all $\phi\in L^2(0,T;H^1(\Omega))$.
Moreover, it contains the homogeneous Neumann boundary conditions
in \eqref{1.bic}

It remains to show that $u(0)$ satisfies the initial datum. Let
$\widetilde u^{(\tau)}$ be the linear interpolant 
$\widetilde u^{(\tau)}(t)=u_k-(k\tau-t)(u^k-u^{k-1})/\tau$ for
$(k-1)\tau\le t\le k\tau$, where $u^k=u(w^k)$. Because of \eqref{a.ut},
$$
  \|\pa_t \widetilde u^{(\tau)}\|_{L^2(0,T-\tau;H^m(\Omega)')}
	\le \tau^{-1}\|u^{(\tau)}-\sigma_\tau u^{(\tau)}\|_{L^2(\tau,T;H^m(\Omega)')}
	\le C.
$$
This shows that $(\widetilde u^{(\tau)})$ is bounded in $H^1(0,T;H^m(\Omega)')$. 
Thus, for a subsequence, $\widetilde u^{(\tau)}\rightharpoonup w$
weakly in $H^1(0,T;H^m(\Omega)')\hookrightarrow C^0([0,T];H^m(\Omega)')$
and, by weak continuity, $\widetilde u^{(\tau)}(0)\rightharpoonup w(0)$
weakly in $H^m(\Omega;\R^n)'$. However,
$\widetilde u^{(\tau)}$ and $u^{(\tau)}$ converge to the same limit since
$$
  \|\widetilde u^{(\tau)}-u^{(\tau)}\|_{L^2(0,T-\tau;H^m(\Omega)')}
	\le \|u^{(\tau)}-\sigma_\tau u^{(\tau)}\|_{L^2(\tau,T;H^m(\Omega)')}
	\le \tau C\to 0
$$
as $\tau\to 0$. We infer that $w=u$ and $u^0=\widetilde u^{(\tau)}(0)
\rightharpoonup u(0)$ weakly in $H^m(\Omega;\R^n)'$. This shows that
the initial datum is satisfied in the sense of $H^m(\Omega;\R^n)'$ and,
in view of $u\in H^1(0,T;H^1(\Omega)')$, also in $H^1(\Omega;\R^n)'$.


\section{Proof of Theorem \ref{thm.exq}}\label{sec.thm2}

We impose more general assumptions on $q$ than in Theorem \ref{thm.exq}: 
Let $q\in C^1([0,1])$ be positive and nondecreasing on $(0,1)$ such that
\begin{equation}
  q(0)=0, \ \mbox{there exists }\kappa>0\mbox{ such that }
	\tfrac12 yq'(y)\le q(y)\le\kappa q'(y)^2\mbox{ for }y\in[0,1].
	\label{se.assq}
\end{equation}
For instance, the functions $q(y)=y/(1+y)^s$ with $0\le s\le 1$
and $q(y)=y^s$ with $1\le s\le 2$
fulfill \eqref{se.assq}. The case $q(y)=y^s$ with any $s\ge 1$ 
is also possible, see Theorem \ref{thm.exq2}.

\begin{theorem}\label{thm.exq1}
Let $\beta>0$ and let $q\in C^1([0,1])$ be positive and nondecreasing on $(0,1)$.
We assume that \eqref{se.assq} holds and that $A(u)$ is given by \eqref{1.Aq}.
Let $u^0=(u_1^0,u_2^0)\in L^1(\Omega)$ with $0<u_1^0,u_2^0<1$ and $u_1^0+u_2^0<1$.
Then there exists a solution $u$ to \eqref{1.eq}-\eqref{1.bic} with $f=0$ satisfying
\begin{align*}
  & 0\le u_i,u_3\le 1\mbox{ in }\Omega,\ t>0, \quad
	\pa_t u_i\in L^2(0,T;H^1(\Omega)'), \\
	& q(u_3)^{1/2},\,q(u_3)^{1/2}u_i\in L^2(0,T;H^1(\Omega)), \quad i=1,2,
\end{align*}
where $u_3=1-u_1-u_2\ge 0$, the weak formulation \eqref{1.qw} holds for all 
$\phi\in L^2(0,T;H^1(\Omega))$, and $u(0)=u^0$ in the sense of $H^1(\Omega;\R^2)'$.
\end{theorem}

\begin{proof}
First we observe that we may assume that $\beta\ge 1$
since otherwise we rescale the equations by $t\mapsto \beta t$, which removes
this factor in the second equation but yields 
the factor $1/\beta>1$ in the first equation.

{\em Step 1: Verification of Assumptions H1-H2.}
We claim that Assumptions H1 and H2 are satisfied for the entropy density
\begin{equation*}
  h(u) = u_1(\log u_1-1) + u_2(\log u_2-1) + \int_{1/2}^{u_3}\log q(s)ds
\end{equation*}
for $u\in D=\{(u_1,u_2)\in(0,1)^2:u_1+u_2<1\}$. Indeed, the 
Hessian $D^2h$ is positive definite on $D$.
Therefore, $h$ is strictly convex. We claim that $Dh:D\to\R^2$ is invertible.
For this, let $(w_1,w_2)\in\R^2$ and define $g(y)=(e^{w_1}+e^{w_2})q(1-y)$
for $0<y<1$. Since $q$ is nondecreasing, $g$ is nonincreasing. Furthermore,
$g(0)>0$ and $g(1)=0$ using $q(0)=0$. By continuity, 
there exists a unique fixed point $0<y_0<1$ such that $g(y_0)=y_0$. Then we define
$u_i=e^{w_i}q(1-y_0)>0$ ($i=1,2$) which satisfies $u_1+u_2=g(y_0)=y_0<1$.
Consequently, $u=(u_1,u_2)\in D$.
We set $u_3=1-u_1-u_2=1-y_0$ which gives $w_i=\log(u_i/q(u_3))=(\pa h/\pa u_i)(u)$.
Hence, Assumption H1 is satisfied.

To verify Assumption H2, we show the following lemma.

\begin{lemma}[Positive definiteness of $(D^2h)A$]\label{lem.pd1}
The matrix $(D^2h)A$ is positive semi-defi\-nite. Moreover, if
$yq'(y)\le 2q(y)$ holds for all $0\le y\le 1$ (see \eqref{se.assq}), it holds 
\begin{equation}\label{se.M}
  z^\top (D^2h)A z \ge \frac{q(u_3)}{u_1}z_1^2 + \frac{q(u_3)}{u_2}z_2^2
	+ \frac{q'(u_3)^2}{q(u_3)}(z_1+z_2)^2 \quad \mbox{for }z=(z_1,z_2)^\top\in\R^2.
\end{equation}
\end{lemma}

\begin{proof}
Set $(D^2h)A=M_0+(\beta-1)M_1$, where
\begin{align*}
  M_0 &= \begin{pmatrix} q(u_3)/u_1 & 0 \\ 0 & q(u_3)/u_2 \end{pmatrix}
	+ \frac{q'(u_3)}{q(u_3)}\big(2q(u_3)+(u_1+u_2)q'(u_3)\big)
	\begin{pmatrix} 1 & 1 \\ 1 & 1 \end{pmatrix}, \\
  M_1 &= \begin{pmatrix} 0 & q'(u_3) \\ q'(u_3) & q(u_3)/u_2 + 2q'(u_3) \end{pmatrix}
	+ \frac{q'(u_3)^2u_2}{q(u_3)}\begin{pmatrix} 1 & 1 \\ 1 & 1 \end{pmatrix}.
\end{align*}
A straightforward computation shows that 
\begin{equation*}
\begin{aligned}
  z^\top M_0z &= \frac{q(u_3)}{u_1}z_1^2 + \frac{q(u_3)}{u_2}z_2^2
	+ \frac{q'(u_3)}{q(u_3)}\big(2q(u_3)+(1-u_3)q'(u_3)\big)(z_1+z_2)^2, \\
  z^\top M_1z &= u_2q(u_3)\left(\frac{q'(u_3)}{q(u_3)}z_1 + 
	\left(\frac{1}{u_2}+\frac{q'(u_3)}{q(u_3)}\right)z_2\right)^2 \ge 0
\end{aligned}
\end{equation*}
for all $z=(z_1,z_2)^\top\in\R^2$. If $yq'(y)\le 2q(y)$ holds,
we find that $2q(u_3)+(1-u_3)q'(u_3)\ge q'(u_3)$ from which we infer the result.
\end{proof}

{\em Step 2: A priori estimates.}
By Lemma \ref{lem.ex}, there exists a sequence $(w^k)$ of weak solutions
to \eqref{a.reg} satisfying the entropy-dissipation inequality
\eqref{a.edi} with $C_f=0$. Taking into account the identity
$B(w^k)\na w^k=A(u^k)\na u^k$, where $u^k=u(w^k)$, $w^k$ is a solution to 
\begin{align}
  \frac{1}{\tau}\int_\Omega(u^k-u^{k-1})\cdot\phi dx 
	&+ \int_\Omega\na\phi:A(u^k)\na u^kdx \nonumber \\
  &{}+ \eps\int_\Omega\bigg(\sum_{|\alpha|=m}D^\alpha w^k\cdot D^\alpha\phi
	+ w^k\cdot\phi\bigg)dx = 0
	\label{se.epsw}
\end{align}
for all $\phi\in H^m(\Omega;\R^2)$. Furthermore, because of
$$
  \na w^k:B(w^k)\na w^k=\na u^k:(D^2 h)(u^k)A(u^k)\na u^k
$$ 
and \eqref{se.M}, the 
discrete entropy inequality \eqref{a.edi} can be written as
\begin{align}
  \int_\Omega h(u^k)dx &+ 4\tau\int_\Omega q(u_3^k)\sum_{i=1}^2
	|\na(u_i^k)^{1/2}|^2 dx 
	+ 4\tau\int_\Omega |\na q(u_3^k)^{1/2}|^2 dx \nonumber \\
	&{}+ \eps\tau\int_\Omega
	\bigg(\sum_{|\alpha|=m}|D^\alpha w^k|^2 + |w^k|^2\bigg)dx
	\le \int_\Omega h(u^{k-1})dx, \label{se.edi0}
\end{align}
where $u^k=(u_1^k,u_2^k)$ and $u_3^k=1-u_1^k-u_2^k$. Resolving the recursion,
we infer that
\begin{align}
  \int_\Omega h(u^k)dx
	&+ 4\tau\sum_{j=1}^k\int_\Omega q(u_3^j)\sum_{i=1}^2|\na(u_i^j)^{1/2}|^2 dx
	+ 4\tau\sum_{j=1}^k\int_\Omega |\na q(u_3^j)^{1/2}|^2 dx \nonumber \\
	&{}+ \eps\tau\sum_{j=1}^k\int_\Omega
	\bigg(\sum_{|\alpha|=m}|D^\alpha w^j|^2 + |w^j|^2\bigg)dx
	\le \int_\Omega h(u^{0})dx. \label{se.edi}
\end{align}
Using the generalized Poincar\'e inequality \cite[Chap. II.1.4]{Tem97}, we deduce that
\begin{equation}\label{se.wk}
  \eps\tau\sum_{j=1}^k\|w^j_i\|_{H^m(\Omega)}^2 \le C, \quad i=1,2.
\end{equation}
The assumption $q(y)\le\kappa q'(y)^2$ (see \eqref{se.assq}) 
implies an $H^1$ estimate for $u_3^k$:
\begin{align*}
  \tau\sum_{j=1}^k\int_\Omega|\na u_3^j|^2 dx
	\le \kappa\tau\sum_{j=1}^k\int_\Omega \frac{q'(u_3^j)^2}{q(u_3^j)}
	|\na u_3^j|^2 dx\,dt 
	= 4\kappa\tau\sum_{j=1}^k\int_\Omega|\na q(u_3^j)^{1/2}|^2 dx\,dt \le C.
\end{align*}
The boundedness of $(u_i^k)$ in $L^\infty$ yields an $L^2$ estimate
for $(u_i^k)$ and hence, also for $(u_3^k)$. Therefore, 
\begin{equation}\label{se.u3}
  \tau\sum_{j=1}^k\|u_3^j\|_{H^1(\Omega)}^2 \le C.
\end{equation}

We cannot perform the simultaneous limit $(\tau,\eps)\to 0$ as in the proof
of Theorem \ref{thm.ex} since we need a compactness result of the type of
Lemma \ref{lem.comp} (see Appendix \ref{sec.aubin})
in which the discrete time derivative is uniformly bounded in $H^1(\Omega)'$
and not in the larger space $H^m(\Omega)'$ with $m>d/2$. Therefore, we
pass to the limit in two steps.

{\em Step 3: Limit $\eps\to 0$.}
We fix $k\in\{1,\ldots,N\}$ and
set $u_i^{(\eps)}=u_i^k$ for $i=1,2,3$, $w_i^{(\eps)}=w_i^k$ for $i=1,2$,
and $u^{(\eps)}=(u_1^{(\eps)},u_2^{(\eps)})$. 
The uniform $L^\infty$ bounds for $(u_i^{(\eps)})$ and estimates
\eqref{se.wk}-\eqref{se.u3} as well as the compact embedding $H^1(\Omega)
\hookrightarrow L^2(\Omega)$ show that there exist subsequences which are
not relabeled such that, as $\eps\to 0$,
\begin{align*}
  u_i^{(\eps)} \rightharpoonup^* u_i &\quad\mbox{weakly* in }L^\infty(\Omega), 
	\ i=1,2, \\
	u_3^{(\eps)}\rightharpoonup u_3 &\quad\mbox{weakly in }H^1(\Omega), \\
	u_3^{(\eps)}\to u_3 &\quad\mbox{strongly in }L^2(\Omega), \\
  \eps w_i^{(\eps)}\to 0 &\quad\mbox{strongly in }H^m(\Omega).
\end{align*}
Moreover, since $u_3^{(\eps)}=1-u_1^{(\eps)}-u_2^{(\eps)}$, we find that
in the limit $\eps\to 0$,
$u_3=1-u_1-u_2$. Since $u_3^{(\eps)}\to u_3$ a.e.\ in $\Omega$ and $q$ is
continuous, we have $q(u_3^{(\eps)})^{1/2}\to q(u_3)^{1/2}$ a.e.\ in $\Omega$.
The sequence $(q(u_3^{(\eps)})^{1/2})$ is bounded in $L^\infty(\Omega)$. Therefore,
by dominated convergence, 
\begin{equation}\label{se.epsq}
  q(u_3^{(\eps)})^{1/2}\to q(u_3)^{1/2} \quad\mbox{strongly in }L^p(\Omega), 
	\ p<\infty.
\end{equation}
Thus, the $L^2$ bound for $(\na q(u_3^{(\eps)})^{1/2})$
from \eqref{se.edi} shows that, up to a subsequence,
\begin{equation}\label{se.epsnaq}
  \na \big(q(u_3^{(\eps)})^{1/2}\big)\rightharpoonup \na\big(q(u_3)^{1/2}\big)
	\quad\mbox{weakly in }
	L^2(\Omega).
\end{equation}
Furthermore, in view of
$$
  \na\big(q(u_3^{(\eps)})^{1/2}u_i^{(\eps)}\big)
	= u_i^{(\eps)}\na(q(u_3^{(\eps)})^{1/2})
	+ 2(u_i^{(\eps)})^{1/2}q(u_3^{(\eps)})^{1/2}\na\big((u_i^{(\eps)})^{1/2}\big),
$$
estimate \eqref{se.edi} and the $L^\infty$ bounds show that for $i=1,2$,
\begin{align}
  \int_\Omega\big|\na\big(q(u_3^{(\eps)})^{1/2}u_i^{(\eps)}\big)\big|^2 dx
	&\le 2\int_\Omega (u_i^{(\eps)})^2|\na(q(u_3^{(\eps)})^{1/2})|^2 dx 
	+ 8\int_\Omega u_i^{(\eps)}q(u_3^{(\eps)})|\na(u_i^{(\eps)})^{1/2}|^2 dx 
	\nonumber \\
	&\le C\int_\Omega |\na(q(u_3^{(\eps)})^{1/2})|^2 dx 
	+ C\int_\Omega q(u_3^{(\eps)})|\na(u_i^{(\eps)})^{1/2}|^2 dx \le C. 
	\label{se.qu}
\end{align}
It follows that
\begin{equation}\label{se.epsqu}
  \|q(u_3^{(\eps)})^{1/2}u_i^{(\eps)}\|_{H^1(\Omega)} \le C.
\end{equation}
We conclude,
again up to a subsequence, that $q(u_3^{(\eps)})^{1/2}u_i^{(\eps)}$ 
converges weakly in $H^1(\Omega)$ and, because of the compact embedding
$H^1(\Omega)\hookrightarrow L^2(\Omega)$, strongly
in $L^2(\Omega)$ (and in fact, in every $L^p(\Omega)$)
to some function $y\in H^1(\Omega)$. As $q(u_3^{(\eps)})^{1/2}
\to q(u_3)^{1/2}$ strongly in $L^p(\Omega)$ and $u_i^{(\eps)}\rightharpoonup u_i$
weakly in $L^p(\Omega)$ for all $p<\infty$, we obtain 
$q(u_3^{(\eps)})^{1/2}u_i^{(\eps)}\rightharpoonup q(u_3)^{1/2}u_i$ 
weakly in $L^p(\Omega)$. Consequently, $y=q(u_3)^{1/2}u_i$ and
\begin{align}
  q(u_3^{(\eps)})^{1/2}u_i^{(\eps)}\to q(u_3)^{1/2}u_i 
	&\quad\mbox{strongly in }L^p(\Omega), \ p<\infty, \label{se.epsqu2} \\
  \na\big(q(u_3^{(\eps)})^{1/2}u_i^{(\eps)}\big)\rightharpoonup
	\na\big(q(u_3)^{1/2}u_i\big) &\quad\mbox{weakly in }L^2(\Omega). \label{se.epsqu3}
\end{align}

We wish to pass to the limit $\eps\to 0$ in
\begin{align*}
  (A(u^{(\eps)})\na u^{(\eps)})_i
	&= q(u_3^{(\eps)})\na u_i^{(\eps)} + u_i^{(\eps)}q'(u_3^{(\eps)})
	\na(1-u_3^{(\eps)}) \\
	&= q(u_3^{(\eps)})^{1/2}\na\big(q(u_3^{(\eps)})^{1/2}u_i^{(\eps)}\big)
	-3q(u_3^{(\eps)})^{1/2}u_i^{(\eps)}\na\big(q(u_3^{(\eps)})^{1/2}\big).
\end{align*}
Taking into account \eqref{se.epsq} and \eqref{se.epsqu3}, the first summand converges
weakly in $L^1(\Omega)$ to $q(u_3)^{1/2}$ $\times\na(q(u_3)^{1/2}u_i)$. 
By \eqref{se.epsnaq}
and \eqref{se.epsqu2}, the second summand converges weakly in $L^1(\Omega)$
to $-3q(u_3)^{1/2}u_i\na(q(u_3)^{1/2})$. Thus, performing the limit $\eps\to 0$
in \eqref{se.epsw} and setting $u_i^k:=u_i$ ($i=1,2,3$), 
we find that $u=(u_1,u_2)$ solves
\begin{align*}
  \frac{1}{\tau}\int_\Omega(u^k-u^{k-1})\cdot\phi dx
	+ \sum_{i=1}^2\beta_i\int_\Omega q(u_3^k)^{1/2}
	\big(\na(q(u_3^k)^{1/2}u_i^k)-3u_i^k\na(q(u_3^k)^{1/2})\big)\cdot\na\phi_i dx
	= 0
\end{align*}
for all $\phi=(\phi_1,\phi_2)\in H^{m}(\Omega;\R^2)$, where we recall that
$\beta_1=1$ and $\beta_2=\beta$. 
By a density argument, this equation also holds for all
$\phi\in H^1(\Omega;\R^2)$. Because of the weak convergence \eqref{se.epsqu3}
and estimate \eqref{se.qu},
the limit $q(u_3^k)^{1/2}u_i^k$ satisfies the bound
\begin{equation}\label{se.tauqu}
  \tau\sum_{j=1}^k\|q(u_3^j)^{1/2}u_i^j\|_{H^1(\Omega)}^2 \le C,
\end{equation}
where $C>0$ depends on $T_k=k\tau$ but not on $\tau$. Furthermore, 
by \eqref{se.edi} and \eqref{se.u3},
\begin{equation}\label{se.tauu}
  \tau\sum_{j=1}^k\|u_3^j\|_{H^1(\Omega)}^2 \le C.
\end{equation}

{\em Step 4: Limit $\tau\to 0$.}
Let $u^{(\tau)}(x,t)=u^k(x)$ for $x\in\Omega$ and $t\in((k-1)\tau,k\tau]$, 
$k=1,\ldots,N$. At time $t=0$, we set $u^{(\tau)}(\cdot,0)=u^0$. 
Then $u^{(\tau)}$ solves the equation
\begin{align}
  \frac{1}{\tau}&\int_\tau^T\int_\Omega (u^{(\tau)}-\sigma_\tau u^{(\tau)})
	\phi dxdt \nonumber \\
	&{}+ \sum_{i=1}^2\beta_i\int_\tau^T\int_\Omega q(u_3^{(\tau)})^{1/2}
	\big(\na(q(u_3^{(\tau)})^{1/2}u_i^{(\tau)})
	-3u_i^{(\tau)}\na(q(u_3^{(\tau)})^{1/2})\big)\cdot\na\phi_i dxdt = 0
	\label{se.weaktau}
\end{align}
for all $\phi=(\phi_1,\phi_2)\in L^2(0,T;H^1(\Omega;\R^2))$ and 
estimates \eqref{se.tauqu} and \eqref{se.tauu} give
\begin{equation}\label{se.tau}
  \|q(u_3^{(\tau)})^{1/2}u_i^{(\tau)}\|_{L^2(0,T;H^1(\Omega))}
	+ \|u_3^{(\tau)}\|_{L^2(0,T;H^1(\Omega))} \le C, \quad i=1,2,
\end{equation}
where $C>0$ is independent of $\tau$.
Clearly, we still have the $L^\infty$ bounds:
$$
  \|u_i^{(\tau)}\|_{L^\infty(0,T;L^\infty(\Omega))}
	+ \|u_3^{(\tau)}\|_{L^\infty(0,T;L^\infty(\Omega))} \le C.
$$
We claim that the discrete time derivative of $u_3^{(\tau)}$ is also uniformly 
bounded. Indeed, since
\begin{align*}
  \int_0^T & \big\|q(u_3^{(\tau)})^{1/2}\big[\na(q(u_3^{(\tau)})^{1/2}u_i^{(\tau)})
	-3u_i^{(\tau)}\na(q(u_3^{(\tau)})^{1/2})\big]\big\|_{L^2(\Omega)}^2 dt \\
	&\le C\int_0^T\|\na(q(u_3^{(\tau)})^{1/2}u_i^{(\tau)})\|_{L^2(\Omega)}^2 dt
	+ C\int_0^T\|\na(q(u_3^{(\tau)})^{1/2})\|_{L^2(\Omega)}^2 dt \le C,
\end{align*}
we find that
\begin{align}
  \tau^{-1}&\|u_i^{(\tau)}-\sigma_\tau u_i^{(\tau)}\|_{L^2(0,T;H^1(\Omega)')} 
	\nonumber \\
	&\le C\|q(u_3^{(\tau)})^{1/2}(\na(q(u_3^{(\tau)})^{1/2}u_i^{(\tau)})
	-3u_i^{(\tau)}\na(q(u_3^{(\tau)})^{1/2}))\|_{L^2(0,T;L^2(\Omega))}
	\le C, \label{se.ut}
\end{align}
which immediately yields 
\begin{equation}\label{se.u3t}
  \tau^{-1}\|u_3^{(\tau)}-\sigma_\tau u_3^{(\tau)}\|_{L^2(\tau,T;H^1(\Omega)')} \le C.
\end{equation}

Then \eqref{se.tau} and \eqref{se.u3t} allow us to apply the
Aubin lemma in the version of \cite{DrJu12} to conclude (up to a subsequence)
the convergence
$$
  u_3^{(\tau)} \to u_3 \quad\mbox{strongly in }L^2(0,T;L^2(\Omega))
$$
as $\tau\to 0$. Since $(q(u_3^{(\tau)})^{1/2})$ is bounded in $L^\infty$ and
$q(u_3^{(\tau)})^{1/2}\to q(u_3)^{1/2}$ a.e., the dominated 
convergence theorem implies that 
$$
  q(u_3^{(\tau)})^{1/2}\to q(u_3)^{1/2} \quad\mbox{strongly in }L^p(0,T;L^p(\Omega)),
	\quad p<\infty.
$$
In particular, $(q(u_3^{(\tau)})^{1/2})$ is relatively compact in 
$L^2(0,T;L^2(\Omega))$. Furthermore, by \eqref{se.tau} and \eqref{se.ut},
\begin{align*}
  q(u_3^{(\tau)})^{1/2} \rightharpoonup
	q(u_3)^{1/2} &\quad\mbox{weakly in }L^2(0,T;H^1(\Omega)), \\
	\tau^{-1}(u_i^{(\tau)}-\sigma_\tau u_i^{(\tau)})
	\rightharpoonup \pa_t u_i &\quad\mbox{weakly in }L^2(0,T;H^1(\Omega)'),
	\ i=1,2.
\end{align*}
Since $(u_i^{(\tau)})$ converges weakly in $L^p(0,T;L^p(\Omega))$ 
and $(q(u_3^{(\tau)})^{1/2})$ converges strongly in $L^p(0,T;$ $L^p(\Omega))$
for all $p<\infty$, we have
$$
  q(u_3^{(\tau)})^{1/2}u_i^{(\tau)} \rightharpoonup q(u_3)^{1/2}u_i
	\quad\mbox{weakly in }L^p(0,T;L^p(\Omega)), \ p<\infty.
$$
In fact, we can prove that this convergence is even strong. Indeed, 
applying Lemma \ref{lem.comp} in Appendix \ref{sec.aubin}
to $y_\tau=q(u_3^{(\tau)})^{1/2}$ and
$z_\tau=u_i^{(\tau)}$, we infer that, up to a subsequence,
$$
  q(u_3^{(\tau)})^{1/2}u_i^{(\tau)} \to q(u_3)^{1/2}u_i
	\quad\mbox{strongly in }L^p(0,T;L^p(\Omega)),\ p<\infty.
$$
The above convergence results show that
\begin{align*}
  q(u_3^{(\tau)})^{1/2} &
	\na(q(u_3^{(\tau)})^{1/2}u_i^{(\tau)})
	-3q(u_3^{(\tau)})^{1/2}u_i^{(\tau)}\na(q(u_3^{(\tau)})^{1/2})\big) \\
	&\rightharpoonup q(u_i)^{1/2}\na(q(u_3)^{1/3}u_i)
	- 3q(u_3)^{1/2}u_i\na(q(u_3)^{1/2})
\end{align*}
weakly in $L^1(0,T;L^1(\Omega))$. Thus, passing to the limit $\tau\to 0$
in \eqref{se.weaktau} yields \eqref{1.qw}. This finishes the proof.
\end{proof}

\begin{remark}\label{rem.f}\rm
We may allow for nonvanishing
reaction terms $f(u)$ in \eqref{1.eq} if $f$ depends linearly
on $u_1$ and $u_2$ and (possibly) nonlinearly on $u_3=1-u_1-u_2$, since
in the above proof $(u_1^{(\tau)})$ and $(u_2^{(\tau)})$ converge weakly
in $L^p$, whereas $(u_3^{(\tau)})$ converges strongly in $L^p$ ($p<\infty$).
\qed
\end{remark}

Functions $q(y)=y^s$ with $s>2$ can be also considered.

\begin{theorem}\label{thm.exq2}
Let $\beta>1$ and $q(y)=y^s$ with $s\ge 1$.
Let $u^0\in L^1(\Omega)$ with $0<u^0<1$ in $\Omega$.
Then there exists a weak solution to \eqref{1.eq}-\eqref{1.bic} satisfying 
the weak formulation \eqref{1.qw}, $u(0)=u^0$ in $H^1(\Omega;\R^2)'$, and 
\begin{align*}
  & 0\le u_i,u_3\le 1\mbox{ in }\Omega,\ t>0, \quad
	\pa_t u_i\in L^2(0,T;H^1(\Omega)'), \\
  & u_3^{\alpha/2},\ u_3^{\alpha/2}u_i\in L^2(0,T;H^1(\Omega)), \quad i=1,2.
\end{align*}
\end{theorem}

\begin{proof}
It is sufficient to consider the case $s>2$ since the case $s\le 2$ is contained
in Theorem \ref{thm.exq1}.
Assumptions H1 and H2 hold since this requires only $q(0)=0$ and $q'(s)\ge 0$.
By Lemma \ref{lem.ex}, there exists a sequence $(w^k)$ of weak solutions
to \eqref{a.reg} satisfying the entropy-dissipation inequality
\eqref{se.edi0}, which reads as
\begin{align*}
  \int_\Omega h(u^k)dx &+ 4\tau\int_\Omega (u_3^k)^s\sum_{i=1}^2
	|\na (u_i^k)^{1/2}|^2 dx 
	+ s\tau\int_\Omega((2-s)u_3^k+s)(u_3^k)^{s-2}|\na u_3^k|^2 dx \\
	&{}+ \eps\tau\int_\Omega\bigg(\sum_{|\alpha|=m}|D^\alpha w^k|^2 + |w^k|^2\bigg)dx
	\le \int_\Omega h(u^{k-1})dx.
\end{align*}
Observing that 
\begin{align*}
  ((2-s)u_3^k+s)(u_3^k)^{s-2}|\na u_3^k|^2
	&= \frac{4}{s^2}(2u_3^k+s(1-u_3^k))|\na (u_3^k)^{s/2}|^2 \\
	&\ge \frac{4}{s^2}\min\{2,s\}|\na(u_3^k)^{s/2}|^2,
\end{align*}
we infer after summation the uniform estimates
$$
  \tau\sum_{j=1}^k\sum_{i=1}^2\|(u_3^j)^{s/2}\na(u_i^j)^{1/2}
	\|_{L^2(\Omega)}^2
	+ \tau\sum_{j=1}^k\|(u_3^j)^{s/2}\|_{H^1(\Omega)}^2
	+ \eps\tau\sum_{j=1}^k\|w^j\|_{H^1(\Omega)}^2 \le C.
$$
Compared to the proof of Theorem \ref{thm.ex}, we cannot conclude an $H^1$ bound
for $u_3^k$ but only for $(u_3^k)^{s/2}$. Set $u_i^{(\eps)}=u_i^k$ for $i=1,2,3$.
The function $z\mapsto z^{2/s}$ for $z\ge 0$ is H\"older continuous since
$s>2$. Therefore, by the lemma of Chavent and Jaffre \cite[p.~141]{ChJa86},
$$
  \|u_3^{(\eps)}\|_{W^{2/s,s}(\Omega)}\le C\|(u_3^{(\eps)})^{s/2}\|_{H^1(\Omega)}
	\le C.
$$
Since the embedding $W^{2/s,s}(\Omega)\hookrightarrow L^s(\Omega)$ is compact,
we conclude the existence of a subsequence (not relabeled) such that
$$
  u_3^{(\eps)}\to u_3 \quad\mbox{strongly in }L^s(\Omega).
$$
At this point, we can proceed as in Step 3 of the proof of Theorem \ref{thm.ex}.

For the limit $\tau\to 0$, we need another compactness argument. Let
$u_i^{(\tau)}$ be defined as in the proof of Theorem \ref{thm.ex}.
We have the uniform estimates
\begin{align*}
  \|(u_3^{(\tau)})^{s/2}u_i^{(\tau)}\|_{L^2(0,T;H^1(\Omega))}
	+ \|(u_3^{(\tau)})^{s/2}\|_{L^2(0,T;H^1(\Omega))} &\le C, \\
  \tau^{-1}\|u_3^{(\tau)}-\sigma_\tau u_3^{(\tau)}\|_{L^2(0,T;H^1(\Omega)')} &\le C.
\end{align*}
Then, by the generalization of the Aubin lemma in \cite[Theorem 3]{CJL13},
we obtain, up to a subsequence, as $\tau\to 0$,
$$
  u_3^{(\tau)}\to u_3 \quad\mbox{strongly in }L^s(0,T;L^s(\Omega)).
$$
The remaining proof is exactly as Step 4 of the proof of Theorem \ref{thm.ex}.
\end{proof}

\begin{remark}[Generalization of Theorem \ref{thm.exq2}]\rm
Let $q\in C^1([0,1])$ be positive
and nondecreasing on $(0,1)$ such that there exist $0<\kappa_0<1$ and $\kappa_1>0$
such that for all $0\le z\le 1$,
$$
  2q(y)\ge (y-1+\kappa_0)q'(y)\quad\mbox{and either }q(y)\le\kappa_1 q'(y)^2\mbox{ or }
	q(y)\le\kappa_1 y^{2-s} q'(y)^2,
$$
where $s>2$. This includes $q(y)=y^s$ with $s>2$ ($0\le y\le 1$).
Indeed, let $\kappa_0=2/s<1$ and $\kappa_1=1/s^2$. 
Then $2q(y)\ge (y-1+\kappa_0)q'(y)$ is equivalent to $y\le 1$ 
which is true, and $q(y)=y^s\le\kappa_1 y^{2-s} q'(y)^2=y^s$.
The third term of the entropy-dissipation inequality \eqref{se.edi0} can be
estimated as
\begin{align*}
  \int_\Omega \frac{q'(u_3^k)}{q(u_3^k)}\big(2q(u_3^k)+(1-u_3^k)q'(u_3^k)\big)
	|\na u_3^k|^2 dx
	&\ge \kappa_0\int_\Omega\frac{q'(u_3^k)^2}{q(u_3^k)}|\na u_3^k|^2 dx \\
	&= 4\kappa_0\int_\Omega|\na q(u_3^k)^{1/2}|^2 dx,
\end{align*}
and we conclude the estimate
$$
  \tau\sum_{j=1}^k\|q(u_3^j)^{1/2}\|_{H^1(\Omega)} \le C.
$$
In case $q(y)\le\kappa_1 q'(y)^2$, it follows that
\begin{align*}
  \tau\sum_{j=1}^k\int_\Omega|\na u_3^j|^2 dxdt 
	&\le \kappa_1\tau\sum_{j=1}^k\int_\Omega \frac{q'(u_3^j)^2}{q(u_3^j)}
	|\na u_3^j|^2 dxdt \\
	&= 4\kappa_1\tau\sum_{j=1}^k\int_\Omega|\na q(u_3^j)^{1/2}|^2 dxdt \le C.
\end{align*}
In the other case $q(y)\le\kappa_1 y^{2-s}q'(y)^2$, we find that
$$
  \tau\sum_{j=1}^k\int_\Omega |\na (u_3^j)^{s/2}|^2 dxdt
	\le \frac{s^2}{4}\kappa_1\tau\sum_{j=1}^k\int_\Omega\frac{q'(u_3^j)^2}{q(u_3^j)}
	|\na u_3^j|^2 dxdt \le C.
$$
This allows us to apply the generalized Aubin lemma and we proceed as in the
above proof.
\qed
\end{remark}


\section{Proof of Theorem \ref{thm.exp}}\label{sec.thm3}

We prove a slightly more general result than stated in Theorem \ref{thm.exp}
by allowing for more general functions $p_1$ and $p_2$. We suppose that
\begin{equation}\label{ns.pq}
  p_i(u)=\alpha_{i0}+a_{i1}(u_1)+a_{i2}(u_2), \quad i=1,2,
\end{equation}
where $\alpha_{10}$, $\alpha_{20}>0$ are positive numbers, $a_{12}$, $a_{21}$ 
are continuously differentiable, and $a_{11}$, $a_{22}$ are continuous 
functions on $[0,\infty)$. Assume that there exist constants $\alpha_{ij}>0$ such that
\begin{equation}\label{ns.assab}
  a_{11}(y)\ge \alpha_{11} y^{s}, \quad a_{22}(y)\ge\alpha_{22} y^s, \quad
  a_{12}'(y)\ge\alpha_{12} y^{s-1}, \quad a_{21}'(y)\ge\alpha_{21} y^{s-1} 
\end{equation}
for all $y\ge 0$, where $1<s<4$, and that there exist
$C>0$ and $\sigma<2s(1+1/d)-1$ such that for all $u=(u_1,u_2)\in 
(0, \infty)^2$ and $i=1,2$, it holds that $p_i(u)\le C(1+|u_1|^\sigma+|u_2|^\sigma)$.

Since the entropy density $h(u)$, defined in \eqref{1.hp}, may not fulfill 
Hypothesis H1, we need to regularize:
$$
  h_\eps(u) = h(u) + \eps\big(u_1(\log u_1-1)+u_2(\log u_2-1)\big), \quad \eps>0.
$$
Then
$$
  Dh_\eps(u) = Dh(u) + \eps(\log u_1,\log u_2)^\top,
$$
and the range of $Dh_\eps$ is $\R^2$. The Hessian of $h_\eps$ equals
\begin{equation}\label{ns.Heps}
  H_\eps = D^2 h_\eps(u) = \begin{pmatrix}
	a_{21}'(u_1)/u_1 + \eps/u_1 & 0 \\ 0 & a_{12}'(u_2)/u_2 + \eps/u_2
	\end{pmatrix},
\end{equation}
showing that each component of $Dh_\eps$ is strictly increasing, and thus, 
$(Dh_\eps)^{-1}:\R^2\to D=(0,\infty)^2$ is well defined. Hence, $h_\eps$ fulfills 
Hypothesis H1. 
We also regularize the diffusion matrix by setting
\begin{equation}\label{ns.Aeps}
  A_\eps(u) = A(u)
	+ \eps\begin{pmatrix} u_2 & 0 \\ 0 & u_1 \end{pmatrix}.
\end{equation}
The entropy density is regularized similarly in \cite{DLM13} but
the diffusion matrix is regularized here in a different way.

{\em Step 1: Verification of Hypotheses H1 and H2.} Set $H=D^2h(u)$.
First, we show that $HA$ and $H_\eps A_\eps$ are positive definite
under additional conditions on the functions $a_{ij}$.

\begin{lemma}[Positive definiteness of $HA$ and $H_\eps A_\eps$]\label{lem.pd}
Let $H=D^2 h(u)$, where $h$ is defined by \eqref{1.hp}, and let $A=A(u)$ be given
by \eqref{1.Ap} with $p_i$ as in \eqref{ns.pq}.
If for some $s\ge 1$, $sa_{12}(u_2)\ge u_2a_{12}'(u_2)$, 
$sa_{21}(u_1)\ge u_1a_{21}'(u_1)$, and 
$a_{11}'(u_1)a_{22}'(u_2)\ge (1-1/s)a_{12}'(u_2)a_{21}'(u_1)$ 
for all $u_1$, $u_2\ge 0$, 
then $HA$ is positive definite and for all $z=(z_1,z_2)^\top\in\R^2$,
\begin{equation*}
  z^\top HA z \ge \frac{\alpha_{10}+a_{11}(u_1)}{u_1}a_{21}'(u_1) z_1^2 
	+ \frac{\alpha_{20}+a_{22}(u_2)}{u_2}a_{12}'(u_2) z_2^2.
\end{equation*}
Furthermore,
if additionally $4a_{21}(u_1)\ge u_1a_{21}'(u_1)\ge 0$ and 
$4a_{12}(u_2)\ge u_2 a_{12}'(u_2)\ge 0$ for all
$u_1$, $u_2\ge 0$, then $H_\eps A_\eps$ is positive definite and for all
$z=(z_1,z_2)^\top\in\R^2$,
$$
  z^\top H_\eps A_\eps z \ge z^\top HAz 
	+ \eps\left(\frac{\alpha_{10}+a_{11}(u_1)}{u_1}z_1^2
	+ \frac{\alpha_{20}+a_{22}(u_2)}{u_2}z_2^2\right).
$$
\end{lemma}

If $p_1$ and $p_2$ are given by \eqref{1.pp}, $HA$ is positive definite if $s\ge 1$, 
$\alpha_{11}\alpha_{22}\ge (1-1/s)\alpha_{12}\alpha_{21}$, and $H_\eps A_\eps$
is positive definite if additionally $s\le 4$. The proof also works
in the case $s<1$; in this situation the restriction 
$\alpha_{11}\alpha_{22}\ge (1-1/s)\alpha_{12}\alpha_{21}$ is not needed.

\begin{proof}
We have
$$
  HA = \begin{pmatrix}
	(p_1/u_1 + a_{11}')a_{21}' & a_{12}'a_{21}' \\
  a_{12}'a_{21}' & (p_2/u_2+a_{22}')a_{12}'
	\end{pmatrix}.
$$
Then, for $z=(z_1,z_2)^\top\in\R^2$, 
\begin{align*}
  z^\top HA z
	&= \frac{1}{s}a_{12}'a_{21}'\left(\sqrt{\frac{u_2}{u_1}}z_1+\sqrt{\frac{u_1}{u_2}}z_2
	\right)^2 + \left(\frac{\alpha_{10}+a_{11}}{u_1} 
	+\frac{1}{u_1}\left(a_{12}-\frac{1}{s}u_2a_{12}'\right)\right)a_{21}' z_1^2 \\
	&\phantom{xx}{}
	+ \left(\frac{\alpha_{20}+a_{22}}{u_2}
	+\frac{1}{u_2}\left(a_{21}-\frac{1}{s}u_1a_{21}'\right)
	\right)a_{12}' z_2^2 \\
	&\phantom{xx}{}
	+ 2\left(1-\frac{1}{s}\right)a_{21}'a_{12}' z_1z_2 + a_{11}'a_{21}'z_1^2 
	+ a_{12}'a_{22}'z_2^2.
\end{align*}
The last three terms are nonnegative for all $z_1$, $z_2\in\R$ if and only if
$a_{11}'a_{21}'\ge 0$ and $a_{12}'a_{22}'\ge (1-1/s)a_{21}'a_{12}'$. This shows the first statement
of the lemma. 

Next, we compute
\begin{align*}
  H_\eps A_\eps
	&= HA + \eps\begin{pmatrix} 1/u_1 & 0 \\ 0 & 1/u_2 \end{pmatrix}
	A(u) + \eps H\begin{pmatrix} u_2 & 0 \\ 0 & u_1 \end{pmatrix}
	+ \eps^2\begin{pmatrix} u_2/u_1 & 0 \\ 0 & u_1/u_2 \end{pmatrix} \\
	&= HA + \eps\begin{pmatrix} 
	p_1/u_1+a_{11}'+a_{21}'(u_2/u_1) & a_{12}' \\ a_{21}' & p_2/u_2+a_{22}'+a_{12}'(u_1/u_2) \end{pmatrix} \\
	&\phantom{xx}{}+ \eps^2\begin{pmatrix} u_2/u_1 & 0 \\ 0 & u_1/u_2 \end{pmatrix}.
\end{align*}
The first and the third matrix on the right-hand side are positive definite.
Therefore, we need to analyze only the second matrix, called $M$:
\begin{align*}
  z^\top Mz 
	&= \left(\frac{p_1}{u_1}+a_{11}'+a_{21}'\frac{u_2}{u_1}\right)z_1^2
	+ \left(\frac{p_2}{u_2}+a_{22}'+a_{12}'\frac{u_1}{u_2}\right)z_2^2
  + (a_{21}'+a_{12}')z_1z_2 \\
	&= \frac12 a_{21}'\left(\sqrt{\frac{2u_2}{u_1}}z_1+\sqrt{\frac{u_1}{2u_2}}z_2
	\right)^2 
	+ \frac12 a_{12}'\left(\sqrt{\frac{u_2}{2u_1}}z_1+\sqrt{\frac{2u_1}{u_2}}z_2\right)^2 \\
	&\phantom{xx}{}+ \left(\frac{\alpha_{10}+a_{11}+a_{12}}{u_1}+ a_{11}'+a_{21}'\frac{u_2}{u_1}
	- \left(a_{21}'+\frac14 a_{12}'\right)\frac{u_2}{u_1}\right)z_1^2 \\
	&\phantom{xx}{}+ \left(\frac{\alpha_{20}+a_{21}+a_{22}}{u_2}+ a_{22}'+a_{12}'\frac{u_1}{u_2} 
	- \left(\frac14 a_{21}'+a_{12}'\right)\frac{u_1}{u_2}\right)z_2^2 \\
	&\ge \left(\frac{\alpha_{10}+a_{11}}{u_1}+\frac{1}{u_1}\left(a_{12}-\frac14 u_2a_{12}'\right)
	\right)z_1^2 
	+ \left(\frac{\alpha_{20}+a_{22}}{u_2}+\frac{1}{u_2}\left(a_{21}-\frac14 u_1a_{21}'
	\right)\right)z_2^2 \\
	&\ge \frac{\alpha_{10}+a_{11}}{u_1}z_1^2 + \frac{\alpha_{20}+a_{22}}{u_2}z_2^2.
\end{align*}
This ends the proof.
\end{proof}

\begin{remark}\rm
The positive semi-definiteness of $HA$ can be proved by only assuming that
$\det A\ge 0$. Indeed, setting $HA=(c_{ij})$ and observing that $HA$ is
symmetric, for  $z=(z_1,z_2)^\top\in\R^2$, 
\begin{align*}
  z^\top HA z 
	&= c_{12}\left(\sqrt[4]{\frac{c_{11}}{c_{22}}}z_1+\sqrt[4]{\frac{c_{22}}{c_{11}}}z_2
	\right)^2 + \sqrt{\frac{c_{11}}{c_{22}}}
	\left(\sqrt{c_{11}c_{22}}-c_{12}\right)z_1^2 \\
	&\phantom{xx}{}
	+ \sqrt{\frac{c_{22}}{c_{11}}}\left(\sqrt{c_{11}c_{22}}-c_{12}\right)z_2^2 \\
	&\ge \left(\sqrt{c_{11}c_{22}}-c_{12}\right)\min\left\{\sqrt{\frac{c_{11}}{c_{22}}},
	\sqrt{\frac{c_{22}}{c_{11}}}\right\}|z|^2,
\end{align*}
and positive semi-definiteness follows if $c_{11}>0$, $c_{22}>0$, and 
$c_{11}c_{22}-c_{12}^2=(\sqrt{c_{11}c_{22}}+c_{12})(\sqrt{c_{11}c_{22}}-c_{12})\ge 0$. 
Now, $c_{11}\ge \alpha_{10}a_{21}'(u_1)/u_1>0$ for all $u_1>0$, $c_{22}>0$ 
for all $u_2>0$, and, setting $\pa_ip_j=\pa p_j/\pa u_i$,
$$
  c_{11}c_{22}-c_{12}^2 = \frac{a_{21}'a_{12}'}{u_1u_2}(p_1+u_1\pa_1p_1)(p_2+u_2\pa_2p_2)
	- (a_{21}'a_{21}')^2
	= \frac{a_{21}'a_{21}'}{u_1u_2}\det A(u),
$$
showing the claim.
\qed
\end{remark}

{\em Step 2: Solution of an approximate problem.}
Set $w=(Dh)^{-1}(u)$. Then $u=u(w)$. The matrix 
$B_\eps(w)=A_\eps(u(w))H_\eps^{-1}(u(w))$ writes as
$$
  B_\eps = \begin{pmatrix}
	u_1p_1/a_1'+u_1^2a_1'/a_{21}'+\eps u_1u_2/a_{21}' & u_1u_2 \\
	u_1u_2 & u_2p_2/a_2'+u_2^2 a_{22}'/a_2'+\eps u_1u_2/a_2'
	\end{pmatrix}.
$$
By Lemma \ref{lem.pd}, $B_\eps$ is positive semi-definite since
$z^\top B_\eps z = (H_\eps^{-1}z)^\top(H_\eps A_\eps)(H_\eps^{-1}z)\ge 0$.
In view of Lemma \ref{lem.ex}, there exists a weak solution $w^k$ to the
approximate problem \eqref{a.reg} satisfying the discrete entropy inequality
\eqref{a.edi} with $C_f=0$. 
Summing this inequality from $j=1,\ldots,k$ and employing the
identity $\na w^k:B_\eps(w^k)\na w^k=\na u^k:H_\eps A_\eps\na u^k$ and 
Lemma \ref{lem.pd}, we find that
\begin{align*}
  \int_\Omega h(u^k)dx &+ \tau\sum_{j=1}^k\int_\Omega\bigg(
	\frac{\alpha_{10}+a_{11}(u^j_1)}{u_1^j}a_{21}'(u_1^j)|\na u_1^j|^2
	+ \frac{\alpha_{20}+a_{22}(u^j_2)}{u_2^j}a_{12}'(u_2^j)|\na u_2^j|^2\bigg)dx \\
	&{}+ \eps\tau\sum_{j=1}^k\int_\Omega
	\bigg(\frac{\alpha_{10}+a_{11}(u_1^j)}{u_1^j}|\na u_1^j|^2
	+ \frac{\alpha_{20}+a_{22}(u_2^j)}{u_2^j}|\na u_2^j|^2\bigg)dx \\
	&{}+ \eps\tau\int_\Omega\bigg(\sum_{|\alpha|=m}|D^\alpha w|^2+|w|^2\bigg)dx
	\le \int_\Omega h(u^{0})dx.
\end{align*}

{\em Step 3: Uniform estimates.}
The growth conditions \eqref{ns.assab} and the above entropy estimate imply that
\begin{align*}
  \tau\sum_{j=1}^k\sum_{i=1}^2
	\int_\Omega\big((u_i^j)^{s-2}+(u_i^j)^{2(s-1)}\big)|\na u_i^j|^2 dx &\le C, \\
	\eps\tau\sum_{j=1}^k\sum_{i=1}^2\int_\Omega
	\big((u_i^j)^{-1}+(u_i^j)^{s-1}\big)|\na u_i^j|^2 dx &\le C.
\end{align*}
Furthermore, since \eqref{ns.assab} shows that $h(u)\ge C(u_1^s+u_2^s-1)$ 
for some $C>0$, the estimate
$$
  \int_\Omega \big((u_1^k)^s + (u_2^k)^s\big) dx \le \int_\Omega h(u^k)dx + C \le C
$$
yields a uniform bound for $u_i^k$ in $L^\infty(0,T;L^s(\Omega))$. 
Note that by the Poincar\'e inequality,
$$
  \|(u_i^k)^s\|_{L^2(0,T;L^2(\Omega))} \le C\|\na(u_i^k)^s\|_{L^2(0,T;L^2(\Omega))}
	+ C\|(u_i^k)^s\|_{L^2(0,T;L^1(\Omega))}\le C,
$$
which implies that $((u_i^k)^s)$ is bounded in $L^2(0,T;H^1(\Omega))$.
Thus, defining $u_i^{(\tau)}$ as in the proof of Theorem \ref{thm.ex},
we have proved the following uniform estimates:
\begin{align}
  \sum_{i=1}^2\|u_i^{(\tau)}\|_{L^\infty(0,T;L^s(\Omega))} &\le C, \label{ns.Ls} \\
	\sum_{i=1}^2\big(\|(u_i^{(\tau)})^{s/2}\|_{L^2(0,T;H^1(\Omega))} 
	+ \|(u_i^{(\tau)})^{s}\|_{L^2(0,T;H^1(\Omega))}\big) &\le C, \label{ns.uH1} \\
	\sqrt{\eps}\sum_{i=1}^2\big(\|(u_i^{(\tau)})^{1/2}\|_{L^2(0,T;H^1(\Omega))} 
	+ \|w_i^{(\tau)}\|_{L^2(0,T;H^m(\Omega))}\big) &\le C. \label{ns.w}
\end{align}

We also need an estimate for the time derivative. Let $r=2s(d+1)/(d(\sigma+1))>1$,
$1/r+1/r'=1$, and
$\phi\in L^{r'}(0,T;X)$, where $X=\{\phi\in W^{m,\infty}(\Omega):\na\phi\cdot\nu=0$ 
on $\pa\Omega\}$.
Then, observing that $\diver(A(u^{(\tau)})\na u^{(\tau)})_i
=\Delta(u_i^{(\tau)}p_i(u_i^{(\tau)}))$, an integration by parts gives
\begin{align*}
  \frac{1}{\tau} \left|\int_0^T\int_\Omega(u_i^{(\tau)}-\sigma_\tau u_i^{(\tau)})
	\phi dxdt\right| 
	&\le \left|\int_0^T\int_\Omega u_i^{(\tau)}p_i(u_i^{(\tau)})\Delta\phi dxdt\right| \\
	&\phantom{xx}{}+ \eps\int_0^T\|w_i^{(\tau)}\|_{H^m(\Omega)}\|\phi\|_{H^m(\Omega)}dt
	\\ 
	&\le \|u_i^{(\tau)}p_i(u^{(\tau)})\|_{L^r(0,T;L^1(\Omega))}
	\|\phi\|_{L^{r'}(0,T;W^{2,\infty}(\Omega))} \\
	&\phantom{xx}{}
	+ \eps\|w_i^{(\tau)}\|_{L^2(0,T;H^m(\Omega))}\|\phi\|_{L^2(0,T;H^m(\Omega))}.
\end{align*}
We estimate the norm of $u_i^{(\tau)}p_i(u^{(\tau)})$. We infer from the
Gagliardo-Nirenberg inequality that $L^2(0,T;H^1(\Omega))\cap
L^\infty(0,T;L^1(\Omega))\hookrightarrow L^{2+2/d}(\Omega_T)$
and thus, by \eqref{ns.Ls} and \eqref{ns.uH1},
$((u_i^{(\tau)})^s)$ is bounded in $L^{2+2/d}(\Omega_T)$, where 
$\Omega_T=\Omega\times(0,T)$. The growth condition on $p_i$ gives the estimate
\begin{align}
  \|u_i^{(\tau)}p_i(u^{(\tau)})\|_{L^r(\Omega_T)}
	&\le C\bigg(1+\sum_{i=1}^2\|u_i^{(\tau)}\|^{\sigma+1}_{L^{(\sigma+1)r}(\Omega_T)}
	\bigg) \nonumber \\
	&= C\bigg(1+\sum_{i=1}^2\|u_i^{(\tau)}\|_{L^{2s(d+1)/d}(\Omega_T)}^{\sigma+1}\bigg)
	\le C. \label{ns.ur}
\end{align}
We conclude that
$$
  \frac{1}{\tau} \left|\int_\tau^T\int_\Omega(u_i^{(\tau)}-\sigma_\tau u_i^{(\tau)})
	\phi dxdt\right| \le C\|\phi\|_{L^2(0,T;H^m(\Omega))}
$$
and
\begin{equation}\label{ns.ut}
  \tau^{-1}\|u^{(\tau)}-\sigma_\tau u^{(\tau)}\|_{L^r(\tau,T;X')} \le C.
\end{equation}

{\em Step 3: The limit $(\tau,\eps)\to 0$.}
By estimates \eqref{ns.uH1}, \eqref{ns.w}, and \eqref{ns.ut}, 
there exists a subsequence which is not relabeled such that, as $(\tau,\eps)\to 0$,
\begin{align*}
  (u_i^{(\tau)})^s \rightharpoonup y &\quad\mbox{weakly in }L^2(0,T;H^1(\Omega)), \\
	\tau^{-1}(u^{(\tau)}-\sigma_\tau u^{(\tau)})
	\rightharpoonup \pa_t u &\quad\mbox{weakly in }L^r(0,T;X'), \\
	\eps w^{(\tau)}\to 0 &\quad\mbox{strongly in }L^2(0,T;H^m(\Omega)).
\end{align*}
Estimates \eqref{ns.uH1} and \eqref{ns.ut} allow us also to apply the Aubin lemma
in the version of \cite{CJL13}, such that, for a subsequence,
$u^{(\tau)}\to u$ strongly in $L^{2s}(0,T;L^{2q}(\Omega))$ as $(\tau,\eps)\to 0$,
where $q\ge 2$ is such that $H^1(\Omega)\hookrightarrow L^q(\Omega)$.
Because of the uniform bound of $(u_i^{(\tau)})$ in $L^{2s(d+1)/d}(\Omega_T)$ and
the a.e.\ convergence of $(u_i^{(\tau)})$, we find that, up to a subsequence,
$u_i^{(\tau)}\to u_i$ strongly in $L^p(\Omega_T)$ for all $p<2s(d+1)/d$ and
in particular for $p=2s$. By \eqref{ns.uH1} and the above strong convergence, 
again up to a subsequence, 
$u_i p_i(u^{(\tau)})\to u_i p_i(u)$ a.e.\ in $\Omega_T$. 
Because of the bound \eqref{ns.ur}, the dominated convergence theorem shows that
$u_i^{(\tau)}p_i(u^{(\tau)})\to u_ip_i(u)$ strongly in $L^r(\Omega_T)$,

Now, let $\phi=(\phi_1,\phi_2)\in L^{r'}(0,T;X)^2$.
After an integration by parts, $u^{(\tau)}$ solves
\begin{align*}
  \int_0^T \langle\tau^{-1}&(\sigma_\tau u^{(\tau)}-u^{(\tau)}),\phi\rangle dt
	+ \sum_{i=1}^2\int_0^T\int_\Omega u_i^{(\tau)} p_i(u^{(\tau)})\Delta\phi_i dxdt \\
	&{}+ \eps\int_0^T\int_\Omega
	\bigg(\sum_{|\alpha|=m}D^\alpha w^{(\tau)}\cdot D^\alpha\phi
	+ w^{(\tau)}\cdot\phi\bigg)dxdt = 0. 
\end{align*}
The above convergence results are sufficient to pass to the limit $(\tau,\eps)\to 0$,
which yields
$$
  \int_0^T\langle\pa_t u,\phi\rangle dt
	+ \sum_{i=1}^2\int_0^T\int_\Omega u_ip_i(u)\Delta\phi dxdt 
	= 0. 
$$
This proves the theorem.


\section{Further results and open problems}\label{sec.open}

In this section, we discuss some further results and open problems related
to cross-diffusion systems and entropy methods.

\begin{enumerate}[(i)]
\item {\em Dirichlet boundary conditions:} 
Dirichlet boundary conditions may be treated under some conditions on
the nonlinearites. To see this, we assume that $u_i=u_{D,i}$ on $\pa\Omega$
for $i=1,\ldots,n$ and we set $w_D=Dh(u_D)$. 
For simplicity, we assume that $u_D=(u_{D,1},\ldots,u_{D,n})$ 
depends on the spatial variable only. Then, using the test function $w-w_D$
in \eqref{1.eq}, we obtain
\begin{align}
  \frac{d}{dt} & \int_\Omega(h(u)-w_D u)dx 
	+ \int_\Omega\na u:(D^2h)(u)A(u)\na u dx \nonumber \\
	&= \int_\Omega\na w_D:A(u)\na u dx + \int_\Omega f(u)\cdot Dh(u) dx
	- \int_\Omega f(u)\cdot w_D dx. \label{op.h}
\end{align}
If $h(u)$ grows superlinearly (such as $u_i\log u_i$) and $w_D$ is smooth, 
we may estimate the first integral on the left-hand side as
$\int_\Omega(h(u)-w_D u)dx\ge \frac12\int_\Omega h(u)dx - C$ for some constant
$C>0$. Under Hypotheses H2'
and H3, the second integral on the left-hand side and the second integral
on the right-hand side are estimated as follows:
\begin{align*}
  \int_\Omega\na u:(D^2h)(u)A(u)\na u dx 
	&\ge C\sum_{i=1}^n\int_\Omega|\na\widetilde\alpha_i(u_i)|^2 dx, \\
  \int_\Omega f(u)\cdot Dh(u) dx &\le C_f\int_\Omega(1+h(u))dx.
\end{align*}
Then, assuming that $A(u)$ and $f(u)$ are such that, for instance,
$$
  |A(u)\na u|\le \sum_{i=1}^n|\na\widetilde\alpha_i(u_i)|, \quad
	|f(u)|\le C(1+h(u)),
$$
we can estimate the remaining integrals on the right-hand side of \eqref{op.h}.
In a similar way, we may treat Robin boundary conditions by combining the
ideas for Dirichlet and homogeneous Neumann conditions.
\item {\em Uniqueness of weak solutions:} It is possible to prove the
uniqueness of weak solutions to \eqref{1.eq}-\eqref{1.bic} (with $f=0$) under the
assumption that $A(u)$ can be written as a gradient, i.e.\ $A(u)=\na\Phi(u)$
for some monotone function $\Phi:D\to\R^n$, by employing
the $H^{-1}$ method. Unfortunately, the monotonicity assumption is rather strong
since it requires that $\na\Phi(u)=A(u)$ is positive semi-definite.

Another uniqueness result can be obtained if we are able to write the
diffusion equation in terms of the entropy variable (see \eqref{1.eqw}),
for instance in the approximated setting. For this result, we need to suppose
that $B(w)=\na\Phi(w)$ and that $\Phi\circ Dh$ is monotone. 
Let $w^{(1)}$ and $w^{(2)}$ be two weak solutions 
to \eqref{1.bic}-\eqref{1.eqw} (with $f=0$) with the same initial data. 
Furthermore, let $v\in L^2(0,T;H^1(\Omega;\R^n))$
be the weak solution to $-\Delta v=u(w^{(1})-u(w^{(2)})$ in $\Omega$ and
$\na v\cdot\nu=0$ on $\pa\Omega$. Then
\begin{align*}
  \frac12\int_\Omega|\na v(t)|^2 dx 
  &= \int_0^T\langle \pa_t(u(w^{(1})-u(w^{(2)})),v\rangle dt \\
	&= -\int_0^T\int_\Omega\na(\Phi(w^{(1)})-\Phi(w^{(2)})):\na v dxdt \\
	&= -\int_0^T\int_\Omega(\Phi(w^{(1)})-\Phi(w^{(2)}))\cdot
	(u(w^{(1)})-u(w^{(2)}))dxdt \\
	&= -\int_0^T\int_\Omega\big(\Phi(Dh(u^{(1)}))-\Phi(Dh(u^{(2)}))\big)\cdot
	(u^{(1)}-u^{(2)})dxdt \le 0,
\end{align*}
by the monotonicity of $\Phi\circ Dh$, where $u^{(i)}=u(w^{(i)})$ for $i=1,2$.
This shows that $v=0$ and hence $u(w^{(1)})=u(w^{(2)})$. 
Uniqueness results under weaker conditions
are an open problem (at least to our best knowledge). In fact, it is well known
that this question is delicate since non-uniqueness is well known even 
for some scalar equations; see, e.g., \cite{HaWe82}.
\item {\em Quadratic reaction terms:} 
In this paper, our focus was rather on the diffusion part than on the
reaction part. The question is whether the results can be extended to 
diffusion systems with, say,
quadratic reaction rates as they arise in reversible chemistry. For instance,
a global existence result for the following problem is generally unknown:
$$
  \pa_t u - \diver(A(u)\na) = f(u)\quad \mbox{in }\Omega, 
	\quad f_i(u)=(-1)^i(u_1u_3-u_2u_4),
$$
for $i=1,2,3,4$, with boundary and initial conditions \eqref{1.bic}.
If $A(u)$ is the diagonal matrix $A=\mbox{diag}(d_1,\ldots,d_n)$
with smooth functions $d_i=d_i(x,t)\ge 0$, the global existence of weak
solutions was shown in \cite{DFPV07}. This result is based on the observation
that $h(u)=\sum_{i=1}^4 u_i(\log u_i-1)$ possesses an $L^2$ estimate. 
This is proved by using the duality method of M.~Pierre \cite{PiSc97}, and
the proof is based on the diagonal structure of $A$. A global existence result 
for a system modeling the more general reaction $A_1+A_2\rightleftharpoons A_3+A_4$
was given in \cite{CDF14} in the two-dimensional case, the three-dimensional case
being an open problem. Reaction-diffusion systems with
diffusivities $d_i$ depending on $u$ and quadratic rate functions were
recently analyzed in \cite{BoRo14}, but still only in the diagonal case.
\item {\em Long-time behavior of solutions:}
We expect that the long-time behavior of weak solutions to \eqref{1.eq}-\eqref{1.bic}
with $f=0$ can be proven under some additional conditions. For specific diffusion
matrices, the exponential decay has been already shown; 
see \cite{JuSt12,JuSt14}. The idea
is to estimate the entropy dissipation in terms of the entropy. For instance,
let $u_\infty=(u_{\infty,1},\ldots,u_{\infty,n})$ be a constant steady-state
to \eqref{1.eq}-\eqref{1.bic}. To simplify, we assume that the entropy is given by 
$$
  {\mathcal H}^*=\sum_{i=1}^n\int_\Omega u_i\log\frac{u_i}{u_\infty}dx,
$$
and that the entropy dissipation can be estimated from below according to
\begin{equation}\label{app.ed}
  \int_\Omega \na w:B(w)\na wdx 
	= \int_\Omega \na u:(D^2h)A(u)\na u dx \ge \lambda{\mathcal H}^*
\end{equation}
for some $\lambda>0$. Then the entropy-dissipation inequality \eqref{1.dHdt} becomes 
$$
  \frac{d{\mathcal H}}{dt} + \lambda{\mathcal H} \le 0,
$$
and Gronwall's lemma implies exponential convergence in terms of the
relative entropy ${\mathcal H}^*$. By the Csisz\'ar-Kullback inequality,
we conclude the exponential decay in the $L^1$ norm, $\|u(t)-u_\infty\|_{L^1(\Omega)}
\le C\exp(-\lambda t/2)$ for $t>0$. For details, we refer to \cite{CJMTU01}.

The main task is to derive the entropy-dissipation relation \eqref{app.ed}.
According to Hypothesis H2', we need to prove
$$
  \sum_{i=1}^n\int_\Omega|\na\widetilde\alpha_i(u_i)|^2 dx
	\ge \lambda\sum_{i=1}^n\int_\Omega u_i\log\frac{u_i}{u_\infty}dx.
$$
If $\widetilde\alpha_i(u_i)=u_i^{1/2}$ (for instance), this inequality is a
consequence of the logarithmic Sobolev inequality. 
For more general functions $\widetilde\alpha_i$ and power-type entropy densities, 
one may employ Beckner-type inequalities; see \cite[Lemma 2]{CJS13}.
The general case, however, is an open problem.
\item {\em Entropies for population models:} 
It is an open problem to find an entropy for the general
population model with diffusion matrix $A=(a_{ij})$ defined in \eqref{app.a}. 
Just a simple combination of the entropies for the systems analyzed in Theorems
\ref{thm.exq} and \ref{thm.exp}, i.e.\ summing \eqref{1.hq} and \eqref{1.hp},
seems to be not sufficient. It is also an open problem if an entropy
for the population system with matrix \eqref{1.Ap}, where $p_1$ and $p_2$
are given by \eqref{1.pp}, exists without restrictions
on the coefficients $\alpha_{ij}$ (except positivity).
\item {\em Gradient systems and geodesic convexity:} There might to be a relation
between our entropy formulation and the gradient structure for reaction-diffusion
systems developed by Liero and Mielke \cite{LiMi13} but no results are known
so far. Using purely differential 
methods, they have shown the geodesic $\lambda$-convexity for particular 
reaction-cross-diffusion systems. It is an open problem if such geometric properties
also hold for the examples presented in Section \ref{sec.ex}.
\item {\em Coupled PDE-ODE problems:} 
Coupled reaction-diffusion-ODE problems occur, for instance, in 
chemotaxis-haptotaxis systems modeling cancer invasion \cite{TaCu10}. 
Denoting by
$u_1$, $u_2$, $u_3$ the densities of the cancer cells, matrix-degrading enzymes,
extracellular matrix, respectively, the evolution is governed by \eqref{1.eq}
with the exemplary diffusion matrix
$$
  A(u) = \begin{pmatrix}
	a_{11}(u) & a_{12}(u) & a_{13}(u) \\ a_{21}(u) & 0 & 0 \\ 0 & 0 & 0
	\end{pmatrix}.
$$
It is an open problem whether these systems possess an entropy structure and 
whether our method can be extended to such problems.
One idea could be to reduce the problem to the subsystem $(u_1,u_2)$ and
to estimate terms involving $\na u_3$ directly by differentiating
the ODE for $u_3$. 
\item {\em Bounded weak solutions:} A challenging task is to determine
all diffusion systems whose weak solutions are bounded. As a first step in this
direction, we consider $n=2$ and diffusion matrices $A(u)=(a_{ij})$
depending linearly on the variables,
$$
  a_{ij}(u) = \alpha_{ij} + \beta_{ij}u_1 + \gamma_{ij}u_2, \quad i,j=1,2,
$$
where $\alpha_{ij}$, $\beta_{ij}$, $\gamma_{ij}\in\R$.
We wish to find conditions on the coefficients for which the logarithmic
entropy density \eqref{1.log} satisfies Hypothesis H2.  
Requiring that $B=A(D^2h)^{-1}$ is symmetric, we can fix some coefficients,
\begin{align*}
  & \alpha_{12} = \alpha_{21} = \gamma_{12} = \beta_{21} = 0, \quad 
	\beta_{22} = \beta_{11}-\gamma_{21}, \\
	&\gamma_{11} = \gamma_{22}-\beta_{12}, \quad 
	\gamma_{21} = \beta_{12}+\alpha_{22}-\alpha_{11}.
\end{align*}
The matrix $(c_{ij})=(D^2h)A$ is positive semi-definite if and only if
\begin{align}
  c_{11} &= \frac{1}{u_1(1-u_1-u_2)}\big((\beta_{12}-\gamma_{22})u_2^2 
	+ (-\beta_{11}+\beta_{12}+\alpha_{22}-\alpha_{11})u_1u_2 \nonumber \\
	&\phantom{xx}{}	+ \beta_{11}u_1
	+ (-\alpha_{11}-\beta_{12}+\gamma_{22})u_2 + \alpha_{11}\big) 
	\ge 0, \label{app.c1} \\
	\det(c_{ij}) &= \frac{1}{u_1u_2(1-u_1-u_2)}\big(
	\beta_{11}(\beta_{11}-\beta_{12}-\alpha_{22}+\alpha_{11})u_1^2 
	- \gamma_{22}(\beta_{12}-\gamma_{22})u_2^2 \nonumber \\
	&\phantom{xx}{}+ (\alpha_{11}\gamma_{22}-\alpha_{22}\gamma_{22}-\beta_{11}\beta_{12}
	+2\beta_{11}\gamma_{22}-\beta_{12}\gamma_{22})u_1u_2 \nonumber \\
	&\phantom{xx}{}
	+ (\alpha_{11}^2-\alpha_{11}\alpha_{22}+\alpha_{11}\beta_{11}-\alpha_{11}\beta_{12}
	+\alpha_{22}\beta_{11})u_1 \nonumber \\
	&\phantom{xx}{}
	+ (\alpha_{11}\gamma_{22}-\alpha_{22}\beta_{12}+\alpha_{22}\gamma_{22})u_2
	+ \alpha_{11}\alpha_{22}\big) \ge 0 \label{app.c2}
\end{align}
holds for all $u_1$, $u_2>0$ such that $u_1+u_2<1$, where admissible 
ranges for the remaining five parameters
$\alpha_{11}$, $\alpha_{22}$, $\beta_{11}$, $\beta_{12}$, $\gamma_{22}$ 
have to be determined. Examples are (i) the volume-filling model of
Burger et al.\ (see Section \ref{sec.vf}) and (ii) the model defined by \eqref{1.ks},
where
\begin{align*}
  \mbox{(i)}\ &\alpha_{11}=1,\ \alpha_{22}=\beta,\ \beta_{11}=0,\ \beta_{12}=1,\ 
	\gamma_{22}=0, \\
  \mbox{(ii)}\ &\alpha_{11}=1,\ \alpha_{22}=1,\ \beta_{11}=-1,\ \beta_{12}=-1,\
	\gamma_{22}=-1.
\end{align*}
Other examples can be easily constructed. For instance, for
$\alpha_{11}=\alpha_{22}=\beta_{11}=\beta_{12}=1$ we have
$$
  A(u) = \begin{pmatrix}
	1+u_1+(c_{22}-1)u_2 & u_1 \\ u_2 & 1+c_{22}v 
	\end{pmatrix},
$$
and $(c_{ij})=(D^2h)A$ is positive semi-definite if and only if
\begin{align*}
  c_{11} &= \frac{(1-u_2)(1+(c_{22}-1)u_2)+u_1}{u_1(1-u_1-u_2)}\ge 0\quad\mbox{and} \\
  \det(c_{ij}) &= (1+(c_{22}-1)u_2)(1+u_1+c_{22}u_2)\ge 0, 
\end{align*}
which is the case if $0\le c_{22}<\infty$. However, it seems to be difficult to solve
inequalities \eqref{app.c1}-\eqref{app.c2} in the general situation.
Possibly, techniques from quadratic optimization with inequality constraints 
and quantifier elimination may help.
\end{enumerate}


\begin{appendix}
\section{Relations to non-equilibrium thermodynamics}\label{sec.thermo}

We show that the entropy variable $w=Dh(u)$, defined in
Section \ref{sec.intro}, is strongly related to the chemical potentials of a fluid 
mixture and that the particular change of unknowns associated with the logarithmic
entropy density \eqref{1.log} is related to a special choice of the
thermodynamic activities.

We introduce first the thermodynamic setting.
Consider a fluid consisting of $N$ components with the same molar mass under
isobaric and isothermal conditions. We write $\rho_i$ instead of $u_i$ to denote
the mass density of the $i$th component. The evolution of the mass densities
is governed by the mass balance equations
\begin{equation}\label{app.rho}
  \pa_t\rho_i + \diver J_i = 0, \quad i=1,\ldots,N,
\end{equation}
where $J_i$ are the diffusion fluxes. 
We have assumed that the barycentric velocity vanishes and that
there are no chemical reactions. Furthermore, we assume for simplicity
that the total mass density is constant, $\sum_{j=1}^N\rho_j=1$.

Let $s(\rho)=s(\rho_1,\ldots,\rho_N)$ be the thermodynamic entropy of the system.
Then the chemical potentials $\mu_i$ are defined (in the isothermal case) by
$$
  \mu_i = -\frac{\pa s}{\pa\rho_i}, \quad i=1,\ldots,N,
$$
where here and in the following we set physical constants (like temperature)
equal to one. Neglecting also body forces and the (irreversible)
stress tensor, the diffusion fluxes $J_i$ can be written as 
(see \cite[Chapter IV, (15)]{deGMaz62} or \cite[Formula (170)]{BoDr14})
\begin{equation}\label{app.J}
  J_i = -\sum_{j=1}^{N-1}L_{ij}\na(\mu_j-\mu_N), \quad i=1,\ldots,N,
\end{equation}
where $L_{ij}$ are some diffusion coefficients such that $(L_{ij})$ is
positive definite. Once an explicit expression for the thermodynamic entropy 
is determined, equations \eqref{app.rho}-\eqref{app.J} are closed.

Now, we explain the relation of the entropy variables to the above setting.
Since $\rho_N=1-\sum_{j=1}^{N-1}\rho_j$, 
we may express the density of the last component
in terms of the others such that we can introduce
$$
  h(\rho_1,\ldots,\rho_{N-1}) 
	:= -s\bigg(\rho_1,\ldots,\rho_{N-1},1-\sum_{j=1}^{N-1}\rho_j\bigg),
$$
and in fact, $h$ corresponds to the (mathematical) entropy introduced in
Section \ref{sec.intro}. With this notation, the entropy variables become
$$
  w_i = \frac{\pa h}{\pa\rho_i}
	= -\frac{\pa s}{\pa\rho_i} + \frac{\pa s}{\rho_N}
	= -(\mu_i-\mu_N), \quad i=1,\ldots,N-1,
$$
which relates the entropy variables to the chemical potentials.
Moreover, comparing the flux vector $J=B(w)\na w$ from \eqref{1.eqw} with
\eqref{app.J}, we see that the diffusion matrix $B(w)$ coincides with $(L_{ij})$.

For the second statement, we recall that the chemical potentials of a mixture
of ideal gases can be formulated as
\begin{equation}\label{app.mu}
  \mu_i = \mu_i^0 + \log\rho_i, \quad i=1,\ldots,N,
\end{equation}
where $\mu_i^0$ is the Gibbs energy which generally depends on temperature and pressure.
Since we have supposed an isobaric, isothermal situation, $\mu_i^0$ is
constant and, for simplicity, we set $\mu_i^0=0$ for $i=1,\ldots,N$.
In order to model non-ideal gases, it is usual in thermodynamics to introduce
the thermodynamic activity $a_i$ and the activity coefficient $\gamma_i$ by
$$
  \mu_i = \mu_i^0 + \log a_i = \log a_i,  \quad\mbox{where } a_i = \gamma_i\rho_i.
$$
If $\gamma_i=1$, we recover the ideal-gas case. In the volume-filling case,
Fuhrmann \cite{Fuh14} has chosen $\gamma_i=1+\sum_{j=1}^{N-1}a_j$
for numerical purposes. Then
$$
  \rho_i = \frac{a_i}{\gamma_i} = \frac{a_i}{1+\sum_{j=1}^{N-1}a_j},
$$
and since $a_i=\exp(\mu_i)$, it follows that
$$
  \rho_i = \frac{e^{\mu_i}}{1+\sum_{j=1}^{N-1}e^{\mu_j}}, \quad i=1,\ldots,N-1.
$$
which corresponds to the inverse transformation \eqref{1.u} if we identify
$\mu_i$ with $w_i$.
This expression can be derived directly from \eqref{app.mu}.
Indeed, if $\mu_i^0=0$, we obtain $w_i=-(\mu_i-\mu_N)=-\log(\rho_i/\rho_N)$
with $\rho_N=1-\sum_{j=1}^{N-1}\rho_j$, and inverting these relations, we find that
$$
  \rho_i = \frac{e^{w_i}}{1+\sum_{j=1}^{N-1}e^{w_j}}, \quad i=1,\ldots,N-1.
$$


\section{Formal derivation of two-species population models}\label{sec.deriv}

We derive formally cross-diffusion systems 
from a master equation for a continuous-time, discrete-space random walk
in the macroscopic limit.
We consider only random walks on one-dimensional lattices but the derivation
extends to higher dimensions in a straightforward way. The lattice is
given by cells $x_i$ ($i\in{\mathbb Z}$) with the uniform cell distance $h>0$.
The densities of the populations
in the $i$th cell at time $t>0$ are denoted by $u_1(x_i,t)$ and $u_2(x_i,t)$, 
respectively. We assume that the population species $u_1$ and $u_2$ move 
from the $i$th cell into the neighboring $(i\pm 1)$th cells with transition rates 
$S_i^\pm$ and $T_i^\pm$, respectively. Then the master equations can be 
formulated as follows \cite{Ost11}: 
\begin{align}
  \pa_t u_1(x_i) &= S_{i-1}^+ u_1(x_{i-1}) + S_{i+1}^- u_1(x_{i+1})
	- (S_i^++S_i^-)u_1(x_i), \label{d.me1} \\
  \pa_t u_2(x_i) &= T_{i-1}^+ u_2(x_{i-1}) + T_{i+1}^- u_2(x_{i+1})
	- (T_i^++T_i^-)u_2(x_i), \label{d.me2}
\end{align}
where $i\in{\mathbb Z}$. We further suppose that the transition rates $S_i^\pm$
and $T_i^\pm$ depend on the departure cell $i$ and the arrival cells $i\pm 1$:
\begin{align*}
  S_i^\pm &= \sigma_0p_1(u_1(x_i),u_2(x_i))q_1(1-u_1(x_{i\pm 1})-u_2(x_{i\pm 1})), \\
	T_i^\pm &= \sigma_0p_2(u_1(x_i),u_2(x_i))q_2(1-u_1(x_{i\pm 1})-u_2(x_{i\pm 1})),
\end{align*}
where $\sigma_0>0$ is some number.
Abbreviating $p_j(x_i)=p_j(u_1(x_i),u_2(x_i))$ and $q_j(x_i)=q_j(1-u_1(x_i)-u_2(x_i))$,
the master equations \eqref{d.me1}-\eqref{d.me2} become
\begin{align}
  \sigma_0^{-1}\pa_t u_1(x_i) 
	&= p_1(x_{i-1})q_1(x_i)u_1(x_{i-1}) + p_1(x_{i+1})q_1(x_i)u_1(x_{i+1}) \nonumber \\
	&\phantom{xx}{}- p_1(x_i)(q_1(x_{i-1})+q_1(x_{i+1}))u_1(x_i), \label{d.aux1} \\
	\sigma_0^{-1}\pa_t u_2(x_i) 
	&= p_2(x_{i-1})q_2(x_i)u_2(x_{i-1}) + p_2(x_{i+1})q_2(x_i)u_2(x_{i+1}) \nonumber \\
	&\phantom{xx}{}- p_2(x_i)(q_2(x_{i-1})+q_2(x_{i+1}))u_2(x_i). \label{d.aux2}
\end{align}
The functions $p_1(x_i)$ and $p_2(x_i)$ model the tendency of the species 
to leave the cell $i$, whereas $q_1(x_i)$ and $q_2(x_i)$ describe the probability 
to move into the cell $i$. 
The latter functions allow us to model the so-called volume-filling effect.
Indeed, we may interpret $u_1$ and $u_2$ as volume fractions satisfying 
$u_1+u_2\le 1$. Then $1-u_1-u_2$ describes the volume fraction not occupied
by the two species. 
If the $i$th cell is fully occupied, i.e.\ $u_1(x_i)+u_2(x_i)=1$, and 
$q_j(1-u_1(x_i)-u_2(x_i))=q_j(0)=0$, the probability to move into the $i$th cell
is zero.

In order to derive a macroscopic model, we perform a formal Taylor expansion
of $u_j(x_{i\pm 1})$, $p_j(x_{i\pm 1})$, and $q_j(x_{i\pm 1})$ around $u_j(x_i)$
up to second order. Furthermore, we assume a diffusive scaling, i.e.\
$\sigma_0=1/h^2$.
Substituting the expansions into \eqref{d.aux1} and
\eqref{d.aux2} and performing the formal limit $h\to 0$, it follows that
(see \cite{Ost11} for details)
\begin{align*}
  \pa_t u_1 &= \pa_x(a_{11}(u)\pa_x u_1 + a_{12}(u)\pa_x u_2), \\
	\pa_t u_2 &= \pa_x(a_{21}(u)\pa_x u_1 + a_{22}(u)\pa_x u_2),
\end{align*}
where $u=(u_1,u_2)$. The diffusion coefficients are given by
\begin{equation}\label{app.a}
\begin{aligned}
  a_{11}(u) &= p_1(u)q_1(u_3)+u_1(\pa_1 p_1(u)q_1(u_3)+p_1(u)q_1'(u_3)), \\
	a_{12}(u) &= u_1(\pa_2 p_1(u)q_1(u_3)+p_1(u)q_1'(u_3)), \\
	a_{21}(u) &= u_2(\pa_1 p_2(u)q_2(u_3)+p_2(u)q_2'(u_3)), \\
	a_{22}(u) &= p_2(u)q_2(u_3) + u_2(\pa_2 p_2(u)q_2(u_3)+p_2(u)q_2'(u_3)),
\end{aligned}
\end{equation}
where $\pa_i p_j=\pa p_j/\pa u_i$ and $u_3=1-u_1-u_2$. 
In several space dimensions, the argumentation is the same but the computations
are more involved. We obtain equation \eqref{1.eq} with $f=0$ and the
diffusion matrix $A$ with coefficients $a_{ij}$ as above.

It seems very difficult---if not impossible---to explore the entropy structure of 
this equation in full generality. Therefore, we investigated two special cases
in this paper.
First, we assumed that $p_1=p_2=1$, $q_1=q$, and $q_2=\beta q$, where $\beta>0$. 
This corresponds to the volume-filling case \eqref{1.Aq}. 
Second, we have set $q_1=q_2=1$. This gives the matrix \eqref{1.Ap}.


\section{A variant of the Aubin compactness lemma}\label{sec.aubin}

\begin{lemma}\label{lem.comp}
Let $(y_\tau)$, $(z_\tau)$ be sequences which are piecewise constant in time with 
step size $\tau>0$ and which are bounded in
$L^\infty(0,T;L^\infty(\Omega))$. Let $(y_\tau)$ be relatively compact in
$L^2(0,T;L^2(\Omega))$, i.e., up to subsequences which are not relabeled,
$y_\tau\to y$ strongly in $L^2(0,T;L^2(\Omega))$ and
$z_\tau\rightharpoonup^* z$ weakly* in $L^\infty(0,T;L^\infty(\Omega))$
as $\tau\to 0$. Finally, let
\begin{align*}
  \|y_\tau\|_{L^2(0,T;H^1(\Omega))}	&\le C, \\
	\|y_\tau z_\tau\|_{L^2(0,T;H^1(\Omega))} 
	+ \tau^{-1}\|\pi_\tau z_\tau-z_\tau\|_{L^2(\tau,T;(H^1(\Omega))')}
	&\le C,
\end{align*}
where $(\pi_\tau z_\tau)(\cdot,t)=z_\tau(\cdot,t+\tau)$ for $0<t\le T-\tau$.
Then there exists a subsequence (not relabeled) such that 
$y_\tau z_\tau\to yz$ strongly in $L^p(0,T;$ $L^p(\Omega))$ for all $p<\infty$.
\end{lemma}

Note that the result would follow from the Aubin compactness lemma if $y_\tau$ 
was bounded from below by a positive constant, since in this situation,
it would suffice to apply the Aubin lemma in the version of 
\cite{DrJu12} to infer the strong convergence of (a subsequence of) $(z_\tau)$ in 
$L^2(0,T;L^2(\Omega))$ which, together with the strong convergence of $(y_\tau)$,
would give the result.

\begin{proof}
The proof is inspired from \cite[Section 4.4]{BDPS10} but parts of the proof
are different. The idea is to prove that
$$
  \lim_{(h,k)\to 0}
	\int_{\Omega_T}\big((y_\tau z_\tau)(x+h,t+k)-(y_\tau z_\tau)(x,t)\big)^2 dx\,dt
	= 0 \quad\mbox{uniformly in }\tau>0,
$$
where $\Omega_T=\Omega\times(0,T)$ and $y_\tau(\cdot,t)$, $z_\tau(\cdot,t)$ 
are extended by zero for $T\le t\le T+k$.
Then the result follows from the lemma of Kolmogorov-Riesz \cite[Theorem 4.26]{Bre11}
and the $L^\infty$ boundedness of $y_\tau z_\tau$. We write
\begin{align*}
  \int_{\Omega_T} & \big((y_\tau z_\tau)(x+h,t+k)-(y_\tau z_\tau)(x,t)\big)^2 d(x,t)\\
	&\le 2\int_{\Omega_T}\big((y_\tau z_\tau)(x+h,t+k)-(y_\tau z_\tau)(x,t+k)
	\big)^2 d(x,t) \\
	&\phantom{xx}{}
	+ 2\int_{\Omega_T}\big((y_\tau z_\tau)(x,t+k)-(y_\tau z_\tau)(x,t)\big)^2 d(x,t)
	= I_1 + I_2.
\end{align*}

For the estimate of $I_1$, we integrate over the line segment $[x,x+h]$
and employ a standard extension operator in $L^2$:
$$
  I_1 \le \int_{\Omega_T}\int_0^1 h^2|\na(y_\tau z_\tau)(x+sh,t+k)|^2 ds\,d(x,t)
	\le Ch^2\|\na(y_\tau z_\tau)\|_{L^2(\Omega_T)} \le Ch^2,
$$
where here and in the following, $C>0$ denotes a generic constant.

For the second integral, we have 
\begin{align*}
  I_2 &\le 4\int_{\Omega_T}(y_\tau(x,t+k)-y_\tau(x,t))^2 z_\tau(x,t+k)^2 d(x,t) \\
	&\phantom{xx}{}+ 4\int_{\Omega_T}y_\tau(x,t)^2(z_\tau(x,t+k)-z_\tau(x,t))^2 d(x,t)
	= I_{21} + I_{22}.
\end{align*}
Since $(z_\tau)$ is bounded in $L^\infty(0,T;L^\infty(\Omega))$, 
we can estimate as follows:
\begin{equation}\label{app.y}
  I_{21} \le C\int_{\Omega_T}(y_\tau(x,t+k)-y_\tau(x,t))^2 d(x,t).
\end{equation}
By assumption, $(y_\tau)$ is relatively compact in $L^2(\Omega_T)$.
By the inverse of the lemma of Kolmogorov-Riesz \cite[Exercise 4.34]{Bre11},
the right-hand side of \eqref{app.y}
converges to zero as $k\to 0$ uniformly in $\tau>0$. Furthermore,
\begin{align*}
  I_{22} &= \int_{\Omega_T}y_\tau(x,t)^2 z_\tau(x,t)(z_\tau(x,t)-z_\tau(x,t+k))
	d(x,t) \\
	&\phantom{xx}{}+ \int_{\Omega_T}y_\tau(x,t+k)^2 z_\tau(x,t+k)
	(z_\tau(x,t+k)-z_\tau(x,t))d(x,t) \\
	&\phantom{xx}{}+ \int_{\Omega_T}(y_\tau(x,t)^2-y_\tau(x,t+k)^2)z_\tau(x,t+k)
	(z_\tau(x,t+k)-z_\tau(x,t))d(x,t) \\
	&= J_1+J_2+J_3.
\end{align*}
Using Lemma 5 in \cite{CJL13} and the bounds on $(y_\tau)$, $(z_\tau)$, 
the first integral can be estimated as
\begin{align*}
  J_1 &\le \|y_\tau^2 z_\tau\|_{L^2(0,T;H^1(\Omega))}
	\|\pi_k z_\tau-z_\tau\|_{L^2(0,T-k;H^1(\Omega)')} \\
  &\le Ck^{1/2}\tau^{-1}\|\pi_\tau z_\tau-z_\tau\|_{L^1(0,T-\tau;H^1(\Omega)')}
	\le Ck^{1/2},
\end{align*}
and this converges to zero as $k\to 0$ uniformly in $\tau>0$.
The same conclusion holds for $J_2$.
Finally, in view of the $L^\infty$ boundedness of $(y_\tau)$ and $(z_\tau)$,
the third integral becomes
\begin{align*}
  J_3 &\le \int_{\Omega_T}(y_\tau(x,t+k)+y_\tau(x,t))(y_\tau(x,t+k)-y_\tau(x,t))
	z_\tau(x,t+k) \\
	&\phantom{xx}{}\times(z_\tau(x,t+k)-z_\tau(x,t))d(x,t) \\
	&\le C\int_{\Omega_T}|y_\tau(x,t+k)-y_\tau(x,t)|d(x,t).
\end{align*}
Because of the relative compactness of $(y_\tau)$ in $L^2$, 
the right-hand side converges to zero uniformly in $\tau>0$. 
This finishes the proof.
\end{proof}


\end{appendix}


\end{document}